\documentclass[12pt]{article} 
\usepackage[colorlinks]{hyperref}
\usepackage{amsmath}
\usepackage{amsthm}
\usepackage{amsfonts}
\usepackage{amssymb}
\usepackage[mathscr]{eucal}
\usepackage{bussproofs}
\usepackage{quiver}
\usepackage{color}
\usepackage{quiver}

\theoremstyle{plain} 
\newtheorem{theorem}{Theorem}[section]

\newtheorem{lemma}[theorem]{Lemma}

\newtheorem{corollary}[theorem]{Corollary}
\theoremstyle{definition}
\newtheorem{definition}[theorem]{Definition}
\newtheorem{example}[theorem]{Example}
\newtheorem{remark}[theorem]{Remark}

\begin{document}

\title{On Geometric Implications}

\author{Amirhossein Akbar Tabatabai\\Institute of Mathematics, Czech Academy of Sciences}

\date{ }

\maketitle

\begin{abstract}
It is a well-known fact that although the poset of open sets of a topological space is a Heyting algebra, its Heyting implication is not necessarily stable under the inverse image of continuous functions and hence is not a geometric concept. This leaves us wondering if there is any stable family of implications that can be safely called geometric. In this paper, we will first recall the abstract notion of implication as a binary modality introduced in \cite{akbar2021implication}. Then, we will use a weaker version of categorical fibrations to define the geometricity of a category of pairs of spaces and implications over a given category of spaces. We will identify the greatest geometric category over the subcategories of open-irreducible (closed-irreducible) maps as a generalization of the usual injective open (closed) maps. Using this identification, we will then characterize all geometric categories over a given category $\mathcal{S}$, provided that $\mathcal{S}$ has some basic closure properties. Specially, we will show that there is no non-trivial geometric category over the full category of spaces. Finally, as the implications we identified are also interesting in their own right, we will spend some time to investigate their algebraic properties. We will first use a Yoneda-type argument to provide a representation theorem, making the implications a part of an adjunction-style pair. Then, we will use this result to provide a Kripke-style representation for any arbitrary implication.\\


\end{abstract}

\section{Introduction}
It is well-known that the poset of open sets of a topological space is a (complete) Heyting algebra, an observation that provides the topological interpretation of the intuitionistic propositional logic $\mathsf{IPC}$. As the interpretation turns out to be complete, one may be tempted to consider $\mathsf{IPC}$ as the logic of spaces in the same way that the classical propositional logic is the logic of the discrete spaces or simply the unstructured sets. Despite the beauty of such a philosophical temptation, the Heyting implication, although present in any of these locales, is not preserved under the inverse image of continuous functions and hence can not be considered a truly geometric concept. (To see more about locales, see \cite{Pultr}).
To solve the issue, one may eliminate the Heyting implication from the language and restrict its expressive power to the so-called coherent fragment. It is also possible to be more faithful to the nature of space (and hence less to the elementary nature of the language) to allow the infinitary disjunctions and achieve the so-called geometric logic \cite{StoneSpaces,TopViaLogic}. In many contexts \cite{vickers2014continuity}, these logics are the natural logical systems to consider and although they seem to be weak at first glance, they prove their sophistication through their natural role in the real practice. For instance, geometric theories play a crucial role in topos theory, where they characterize all Grothendieck topoi as the mathematical universes freely constructed from the free model of geometric theories \cite{maclane2012sheaves,Johnstone1,Johnstone2}. Even the finitary coherent fragment is more powerful than it appears as any classical theory has a coherent conservative extension \cite{dyckhoff2015geometrisation}. 

Having said that, the implication is not what one wants to ignore permanently. Philosophically speaking, implication is the machinery to internalize the meta-relation of the ``\textit{entailment order between the propositions $A$ and $B$}" into a ``\textit{proposition $A \to B$}" to empower the language to talk about its own entailment behavior. The logical realm is full of different instances of implications from the more philosophically motivated conditionals \cite{NuteCross} and the weak implications \cite{Vi2,visser1981propositional,Ru} addressing the impredicativity problem of the Heyting implication to the more mathematically motivated implications such as the ones appeared in provability logic \cite{visser1981propositional} and preservability logic \cite{Iem1,Iem2,LitViss}.

Opening the horizon to the alternative implications, one may wonder if there is any sort of geometric implication, powerful enough to internalize some parts of the structures, on the one hand, and be geometric, on the other. To address such a problem formally, we must first be precise about what we mean by an implication.
Reading the internalization process algebraically, implications are some binary operations over the posets where they internalize the order of the poset, mapping the predicate $a \leq b$ into the element $a \to b$. Naturally, there are many structures and properties to internalize. For instance, the fact that the order is reflexive, i.e., $a \leq a$ internalizes to $a \to a=1$, its transitivity internalizes to $(a \to b) \wedge (b \to c) \leq (a \to c)$ and the existence of the binary meets to $a \to (b \wedge c)=(a \to b) \wedge (a \to c)$.
To provide a definition for a general notion of implication, we must choose the minimum level of internalization to enforce, and we think that the natural minimum property of an order to internalize is simply the fact that it is an order, i.e., that it is reflexive and transitive, see \cite{akbar2021implication}.

\begin{definition}\label{DefImplication} 
Let $\mathcal{A}=(A, \leq, \wedge, \vee, 1, 0)$ be a bounded distributive lattice. A binary operator $\to$ over $\mathcal{A}$, decreasing in its first argument and increasing in its second is called an \emph{implication} over $\mathcal{A}$ if:
\begin{description}
\item[$(i)$]
(\emph{internal reflexivity}) $a \to a=1$, for any $a \in \mathcal{A}$,
\item[$(ii)$]
(\emph{internal transitivity}) $(a \to b) \wedge (b \to c) \leq a \to c$, for any $a, b, c \in \mathcal{A}$.
\end{description}
An implication is called \emph{meet internalizing} if $a \to (b \wedge c)=(a \to b) \wedge (a \to c)$ and \emph{join internalizing} if $(a \vee b) \to c=(a \to c) \wedge (b \to c)$, for any $a, b, c \in \mathcal{A}$.
For any implication, $\neg a$ is an abbreviation for $a \to 0$. If $\to$ is an implication over $\mathcal{A}$, the pair $(\mathcal{A}, \to)$ is called a \emph{strong algebra}. If $\mathcal{A}=\mathcal{O}(X)$, for some space $X$, then the pair $(X, \to)$ is called a \emph{strong space}.
By a \emph{strong algebra map} $f: (\mathcal{A}, \to_{\mathcal{A}}) \to (\mathcal{B}, \to_{\mathcal{B}})$, we mean a bounded lattice map preserving the implication, i.e., $f(a \to_{\mathcal{A}} b)=f(a) \to_{\mathcal{B}} f(b)$, for any $a, b \in \mathcal{A}$. A \emph{strong space map} is a continuous map between spaces such that its inverse image preserves the implication.
\end{definition}

\begin{remark}\label{OtherDefImp} In \cite{akbar2021implication}, it is shown that implications can be equivalently defined as the binary operations over $\mathcal{A}$ satisfying the conditions: 
\begin{description}
\item[$(i')$]
If $a \leq b$ then $a \to b=1$, for any $a, b \in \mathcal{A}$
\item[$(ii)$]
$(a \to b) \wedge (b \to c) \leq a \to c$, for any $a, b, c \in \mathcal{A}$.
\end{description}
\end{remark}

To see a more detailed discussion to motivate the aforementioned definition, the reader may consult \cite{akbar2021implication}. However, it is illuminating to think of an implication as a special case of a general setting in which a category internalizes its hom structure, i.e., its identity and its composition. The general formalization for such a generalized function space is introduced in \cite{DefArrow}, where it is called an arrow. The categorical formalization of the arrows that act as the generalized internal hom functors can be found in \cite{CatArrow}. In this broader story, our implications are nothing but arrows enriched over the category $\{0 \leq 1\}$ rather than $\mathbf{Set}$ and hence they are just the propositional shadows of the more structured arrows.

\begin{example} \label{ExampleOfImplications}
Over any bounded distributive lattice $\mathcal{A}$, there is a \emph{trivial implication} defined by $a \to_t b=1$, for any $a, b \in \mathcal{A}$. The Boolean and the Heyting implications are also implications. To construct a new implication from the old, assume that $(\mathcal{B}, \to_{\mathcal{B}})$ is a strong algebra, $f: \mathcal{A} \to \mathcal{B}$ is an order-preserving map and $g: \mathcal{B} \to \mathcal{A}$ is a finite meet preserving map. Then, it is easy to check that the operator $a \to_{\mathcal{A}} b=g(f(a) \to_{\mathcal{B}} f(b))$ is an implication over $\mathcal{A}$. It is also possible to show that any implication is constructible from the Heyting implication in this way, expanding the base lattice to a locale, see \cite{akbar2021implication}.
\end{example}

Having a definition for implication, it is now reasonable to search for the geometric implications, i.e., the family of implications over the locales of opens of spaces stable under the inverse image of \emph{all} continuous functions. We will show that there is only one such family, namely the one with the trivial implications. To prove that surprising result, we employ a weaker version of the categorical fibrations to develop the relative notion of geometricity of a category $\mathcal{C}$ of strong spaces over a category $\mathcal{S}$ of spaces. Here, geometricity simply means that the implications of the strong spaces in $\mathcal{C}$ are stable under the inverse image of the maps in $\mathcal{S}$. We will then continue by identifying the greatest geometric categories over the subcategories of the open- and closed-irreducible maps. These two families of maps can be considered as the generalizations of the injective open (closed) maps. The implications stable under the open-irreducible maps are the ones for which $c \wedge a \leq b$ implies $c \leq a \to b$. These implications behave similarly to a well-known family of implications called the basic implication introduced in \cite{visser1981propositional} in provability logic and later in  \cite{Ru} for philosophical reasons. For the closed-irreducible maps, the implications are the ones for which $a \leq b \vee c$ implies $(a \to b) \vee c=1$. We will show that having these two properties forces the implications to behave similarly to the Boolean implications as they satisfy the equation $a \to b=\neg a \vee b$. Using these implications and their relationship with the geometricity for the open- and closed-irreducible maps,
we will then identify all the geometric categories over a given category $\mathcal{S}$, provided that $\mathcal{S}$ has some basic closure properties.

Completing the characterization of the geometric categories, as the implications we identified are also interesting in their own right, we will spend the last section to provide a representation theorem for them. We will first use a Yoneda-type argument, making the implications a part of an adjunction-style pair. Then, we will use this result to represent an arbitrary implication as the implication of a topological version of a combination of an intuitionistic Kripke and a neighbourhood frame.

\section{Preliminaries}
In this section, we will recall some basic notions and their corresponding theorems we need throughout the paper.
Let $\mathcal{P}=(P, \leq)$ be a poset.
A subset $S \subseteq P$ is called an \emph{upset} if for any $x, y \in P$, if $x \in S$ and $x \leq y$ then $y \in S$. The \emph{downsets} are defined dually. The set of all upsets of $(P, \leq)$ is denoted by $U(P, \leq)$.
For any $S \subseteq P$, the greatest lower bound of $S$ (resp. the least upper bound of $S$), if it exists, is called the \emph{meet} (resp. \emph{join}) of the elements of $S$ and is denoted by $\bigwedge S$ (resp. $\bigvee S$). If $S=\{a, b\}$, the meet $\bigwedge S$ and the join $\bigvee S$ are denoted by $a \wedge b$ and $a \vee b$, respectively. Moreover, $\bigwedge \varnothing$ and $\bigvee \varnothing$, i.e., the greatest and the least elements of $P$, if exist, are denoted by $1$ and $0$, respectively.
A poset is called a \emph{bounded lattice}, if for any finite subset $S \subseteq P$, both $\bigwedge S$ and $\bigvee S$ exist and it is called \emph{complete} if for any set $S \subseteq P$, the meet $\bigwedge S$ exists.  A bounded lattice is called \emph{distributive}, if $a \wedge (b \vee c)=(a \wedge b) \vee (a \wedge c)$,
for any $a, b, c \in P$. It is called a \emph{locale}, if for any $S \subseteq P$, the join $\bigvee S$ exists and $
a \wedge \bigvee_{b \in S} b =\bigvee_{b \in S} (a \wedge b)$, for any $a \in P$ and $S \subseteq P$. By the \emph{Heyting implication} over a bounded lattice $\mathcal{A}$, we mean the binary operation $\Rightarrow$ over $\mathcal{A}$ such that $a \wedge b \leq c$ iff $a \leq b \Rightarrow c$, for any $a, b, c \in \mathcal{A}$.
A bounded lattice $\mathcal{H}$ is called a \emph{Heyting algebra} if it has the Heyting implication. A bounded lattice $\mathcal{B}$ is called a \emph{Boolean algebra} if all the elements of $\mathcal{B}$ have a complement, i.e., for any $a \in \mathcal{B}$, there is $b \in \mathcal{B}$ such that $a \vee b=1$ and $a \wedge b=0$. 
A subset of a bounded lattice is called a \emph{filter}, if it is an upset and closed under all finite meets. A filter $F$ is called \emph{prime} if $0 \notin F$ and $a \vee b \in F$ implies either $a \in F$ or $b \in F$. The set of all prime filters of a lattice $\mathcal{A}$ is denoted by $\mathcal{F}_p(\mathcal{A})$. A subset of $\mathcal{A}$ is called an \emph{ideal}, if it is a downset and closed under all finite joins. The following theorem is a useful tool when working with bounded distributive lattices:
\begin{theorem}\label{PrimeFilterTheorem}\cite{davey2002introduction,Esakia}(Prime filter theorem)
Let $\mathcal{A}$ be a bounded distributive lattice, $F$ be a filter and $I$ be an ideal such that $F \cap I=\varnothing$. Then, there exists a prime filter $P$ such that $F \subseteq P$ and $P \cap I=\varnothing$.
\end{theorem}
Let $(P, \leq_P)$ and $(Q, \leq_Q)$ be two posets and $f: P \to Q$ be a function. It is called an \emph{order-preserving map}, if it preserves the order, meaning $f(a) \leq_Q f(b)$, for any $a \leq_P b$. An order-preserving map is called an \emph{order embedding} or simply an embedding, if for any $a, b \in P$, the inequality $f(a) \leq_Q f(b)$ implies $a \leq_P b$. An order-preserving map between two bounded lattices (locales) is called a \emph{bounded lattice map} (\emph{locale map}), if it preserves all finite meets and finite joins (arbitrary joins).
For two order-preserving maps $f: P \to Q$ and $g: Q \to P$, the pair $(f, g)$ is called an \emph{adjunction}, denoted by \emph{$f \dashv g$}, if $f(a) \leq_Q b$ is equivalent to $a \leq_P g(b)$, for any $a \in P$ and $b \in Q$. If $f \dashv g$,
the map $f$ is called the \emph{left adjoint} of $g$ and $g$ is called the \emph{right adjoint} of $f$.

\begin{theorem} \cite{Bor1} (Adjoint functor theorem for posets) \label{AdjointFunctorTheorem}
Let $(P, \leq_P)$ be a complete poset and $(Q, \leq_Q)$ be a poset. Then, an order-preserving map $f: (P, \leq_P) \to (Q, \leq_Q)$ has a right (left) adjoint iff it preserves all joins (meets). 
\end{theorem}
Let $X$ be a topological space. We denote the locale of its open subsets by $\mathcal{O}(X)$. A topological space is called $T_0$, if for any two different points $x, y \in X$, there is an open set which contains one of these points and not the other. It is called \emph{Hausdorff} if for any two distinct points $x, y \in X$, there are opens $U, V \in \mathcal{O}(X)$ such that $x \in U$, $y \in V$ and $U \cap V=\varnothing$.
A pair $(X, \leq)$ of a topological space and a partial order is called a \emph{Priestley space} if $X$ is compact and for any $x, y \in X$, if $x \nleq y$, there exists a clopen upset $U$ such that $x \in U$ and $y \notin U$. For any bounded distributive lattice $\mathcal{A}$, the pair $(\mathcal{F}_p(\mathcal{A}), \subseteq)$ is a Priestley space, where $\mathcal{F}_p(\mathcal{A})$ is the set of all prime filters of $\mathcal{A}$ and the topology on $\mathcal{F}_p(\mathcal{A})$ is defined by the basis of the opens in the form $\{P \in \mathcal{F}_p(\mathcal{A}) \mid a \in P \; \text{and} \; b \notin P \}$, for any $a, b \in \mathcal{A}$. Denoting $\{P \in \mathcal{F}_p(\mathcal{A}) \mid a \in P \}$ by $i(a)$, it is known that any clopen upset in this Priestley space equals to $i(a)$, for some $a \in \mathcal{A}$ and any clopen set is in form $\bigcup_{r=1}^n [i(a_r) \cap i(b_r)^c]$, for some finite sets $\{a_1, \ldots, a_n\}, \{b_1, \ldots, b_n\} \subseteq \mathcal{A}$. For a comprehensive explanation, see \cite{davey2002introduction}.

\section{Open, Closed and Weakly Boolean Implications}
In this section, we first introduce the three families of open, closed, and weakly Boolean implications, some of their natural examples and a method to construct the new ones from the old. Then, in Subsection \ref{SubsectionWBIOverSpaces}, we provide a characterization for the weakly Boolean implications defined on the locale of opens of a topological space needed in the next section. Finally, in Subsection \ref{OpenClosedIrreducible}, we introduce two families of continuous maps as the generalized versions of the injective open and closed maps. These families provide the real motivation to consider the above-mentioned families of implications as the classes of the open and closed implications are the greatest classes of implications that are stable under the inverse image of the open- and closed-irreducible maps, respectively. In other words, the conditions we put on any of these two families of implications are necessary if we want them to be stable under the corresponding classes of continuous maps.
\begin{definition}\label{DefOpenClosedImplication}
An implication $\to$ over a bounded distributive lattice $\mathcal{A}$ is called \emph{open} if $a \wedge b \leq c$ implies $a \leq b \to c$, for any $a, b, c \in \mathcal{A}$. It is called \emph{closed} if $a \leq b \vee c$ implies $(a \to b) \vee c=1$, for any $a, b, c \in \mathcal{A}$. An implication is called a \emph{weakly Boolean} implication (WBI, for short), if it is both open and closed. A strong algebra $(\mathcal{A}, \to)$ is called open, closed, or weakly Boolean, if its implication is. 
\end{definition}

\begin{remark}\label{LogicalJustification}
As mentioned before, the real motivation to investigate the three families of implications introduced in Definition \ref{DefOpenClosedImplication} is geometric and will be covered in Subsection \ref{OpenClosedIrreducible}. However, it is also worth providing a logical motivation for them in this remark. For that purpose, first, consider the following sequent-style rule for the classical implication in the usual calculus $\mathbf{LK}$ for classical logic:
\begin{center}
	\begin{tabular}{c c}
	    \AxiomC{$\Gamma, A \Rightarrow B, \Delta$}
		\UnaryInfC{$\Gamma \Rightarrow A \to B, \Delta$}
		\DisplayProof
	\end{tabular}
\end{center}	
The condition for the open implications is half of the adjunction property of Heyting implications and is reminiscent of the above rule, except that in the rule $\Delta$ is considered as empty. The condition for the closed implications is also reminiscent of the above rule. However, this time the restriction changes to the emptiness of $\Gamma$. It is easy to see that an implication is weakly Boolean iff it admits the full rule, i.e., if $c \wedge a \leq b \vee d$ implies $c \leq (a \to b) \vee  d$, for any $a, b, c, d \in \mathcal{A}
$. 
\end{remark}
It is practically helpful to simplify the definition of the open and closed implications from an implication between two inequalities to just one inequality. It is also theoretically important as it shows that the each of these families form a variety. Here is the simplification. 

\begin{lemma}\label{PropertiesOfOpenAndClosed}
Let $(\mathcal{A}, \to)$ be a strong algebra. Then:
\begin{description}
    \item[$(i)$]
$\to$ is open iff $a \leq b \to a \wedge b$, for any $a, b \in \mathcal{A}$. Specially, $a \leq b \to a$ and $a \wedge \neg 1 = a \wedge \neg a$, for any $a, b \in \mathcal{A}$.
    \item[$(ii)$]
$\to$ is closed iff $(a \vee b \to a) \vee b=1$, for any $a, b \in \mathcal{A}$. Specially, $b \vee \neg b=1$, for any $b \in \mathcal{A}$.
  \item[$(iii)$]
  If $\to$ is closed, then $c \wedge a \leq b$ implies $c \leq \neg a \vee b$, for any $a, b, c \in \mathcal{A}$.
\end{description}
\end{lemma}
\begin{proof}
For $(i)$, if an implication is open, as $a \wedge b \leq a \wedge b$, we have $a \leq b \to a \wedge b$. Conversely, if $a \wedge b \leq c$, we have $a \leq b \to (a \wedge b) \leq b \to c$. 
For the special cases, notice that by $a \wedge b \leq a$, we reach $a \leq b \to a$. To prove $a \wedge \neg a=a \wedge \neg 1$, since $a \leq 1$, we reach $\neg 1 \leq \neg a$ which itself implies $a \wedge \neg 1 \leq a \wedge \neg a$. For the converse, note that as $a \leq 1 \to a$, we have $a \wedge \neg a \leq (1 \to a) \wedge (a \to 0) \leq 1 \to 0=\neg 1$. Hence, $a \wedge \neg a \leq a \wedge \neg 1$.
For $(ii)$, its first part is similar to that of $(i)$. For the special case, by setting $a=0$, we have $b \vee \neg b=1$. For $(iii)$, if $c \wedge a \leq b$, then 
$c \leq (c \vee \neg a) \wedge (a \vee \neg a)=(c \wedge a) \vee \neg a \leq \neg a \vee b$.
\end{proof}

\begin{example}\label{BasicExample}
The trivial implication $\to_t$ defined as the constant function $1$ is both open and closed. The Heyting implication is clearly open, while it is not necessarily closed. In fact, it is closed iff it is Boolean. 
To have a closed implication that is not open, we refer the reader to Example \ref{ExamOfSatndardFrame}, where a general machinery to construct the open and closed implications are provided. Here, however, we provide an example of an implication that is neither open nor closed. Let $\mathcal{A}$ be $\mathcal{O}(\mathbb{R})$ and $\Rightarrow$ be its Heyting implication. Now, putting $f(U)=U+1=\{x+1 \mid x \in U\}$ and $g=id_{\mathcal{A}}$ in Example \ref{ExampleOfImplications}, the operation $U \to V=[(U+1) \Rightarrow (V+1)]$ is an implication. However, it is not open as $[\mathbb{R} \to (0, 1)]= [\mathbb{R} \Rightarrow (1, 2)]=(1, 2) \nsupseteq (0, 1)$ and it is not closed as $(-\infty, 0) \cup [(-\infty, 0) \to \varnothing]=(-\infty, 0) \cup [(-\infty, 1) \Rightarrow \varnothing]=(-\infty, 0) \cup (1, \infty) \neq \mathbb{R}$. Finally, to provide a family of implications that are both open and closed, over any Boolean algebra $\mathcal{B}$ define $a \to b=\bar{a} \vee b \vee m$, where $\bar{a}$ is the complement of $a$ and $m \in \mathcal{B}$ is a fixed element. It is easy to see that $\to$ is a WBI.
\end{example}

In the following definition, we provide a combination of an intuitionistic Kripke frame and a neighbourhood frame, serving as an order-theoretic and hence concrete machinery to construct different families of implications. Later, in Section \ref{SecRepresentation}, we will see that these frames are powerful enough to represent all possible implications.

\begin{definition}
A \emph{Kripke-Neighbourhood frame} (KN-frame, for short) is a tuple $\mathcal{K}=(K, \leq, R, B, N)$ of a poset $(K, \leq)$, a binary relation $R$ on $K$, a set $B \subseteq P(X)$ and a map $N: X \to P(U(K, \leq, B))$, where $U(K, \leq, B)$ is the set of all upsets in $B$, such that:
\begin{itemize}
\item[$\bullet$]
$R$ is compatible with the order, i.e., if $x \leq y$ and $(y, z) \in R$, then $(x, z) \in R$,
\item[$\bullet$]
for any $x \in X$ and any upsets $U, V \in B$, if $U \subseteq V$ and $U \in N(x)$ then $V \in N(x)$,
\item[$\bullet$]
$B$ is closed under finite union (including $K$ as the nullary union), complement and the operation $\lozenge_R$ defined by $\lozenge_R(U)=\{x \in X \mid \exists y \in U, (x, y) \in R\}$,
\item[$\bullet$]
$j(U)=\{x \in X \mid U \in N(x)\}$ is in $B$, for any upset $U \in B$.
\end{itemize}
A KN-frame is called \emph{full} if $B=P(X)$. It is called \emph{standard} if $B=P(X)$ and $N(k)=\{U \in U(K, \leq) \mid k \in U\}$. We denote a standard KN-frame by $(K, \leq, R)$ as its only non-trivial ingredients. A KN-frame is called \emph{open} when for any $x, y \in K$ and any upsets $U, V \in B$, if $x \in U$, $(x, y) 
\in R$ and $V \in N(y)$, then $U \cap V \in N(y)$. It is called \emph{closed} when for any $x, y \in K$ and any upsets $U, V \in B$, if $x \notin U$, $(x, y) 
\in R$ and $U \cup V \in N(y)$, then $V \in N(y)$.
\end{definition}

\begin{example}\label{ExamOfFrame}
Let $\mathcal{K}=(K, \leq, R, B, N)$ be a KN-frame. Then, the bounded distributive lattice $U(K, \leq, B)$ of the upsets in $B$ is closed under the operation $U \to_{\mathcal{K}} V=\{x \in K \mid \forall y \in K \, [(x, y) \in R \; \text{and} \; U \in N(y) \; \text{then} \; V \in N(y)]\}$ and the pair $\mathfrak{A}(\mathcal{K})=(U(K, \leq, B), \to_{\mathcal{K}})$ is a strong algebra. Moreover, if $\mathcal{K}$ is open (closed), then so is $\mathfrak{A}(\mathcal{K})$.
To prove the closure of $U(K, \leq, B)$ under the operation $\to_{\mathcal{K}}$, let $U$ and $V$ be two upsets in $B$. Then, notice that $(U \to_{\mathcal{K}} V)^c=\Diamond_R (j(U) \cap j(V)^c)$. As $B$ is closed under complement, finite intersection and $\Diamond_R$ and $j$ maps the upsets in $B$ to the elements of $B$, we can conclude that $(U \to_{\mathcal{K}} V)^c$ and hence $U \to_{\mathcal{K}} V$ is in $B$. Also, using the compatibility of the order with $R$, it is easy to see that $U \to_{\mathcal{K}} V$ is an upset. Hence, $U \to_{\mathcal{K}} V \in U(K, \leq, B)$.\\
To prove that $\to_{\mathcal{K}}$ is an implication, the only non-trivial part is to prove that if $U \subseteq V$, then $U \to_{\mathcal{K}} V=K$, for any upsets $U, V \in B$. Let $x \in K$ be an arbitrary element and assume $(x, y) \in R$ and $U \in N(y)$. As $U \subseteq V$ and $N(y)$ is upward closed for the upsets in $B$, we reach $V \in N(y)$. Therefore, $x \in U \to_{\mathcal{K}} V$.\\
Finally, for the open and closed conditions, if $\mathcal{K}$ is open, using Lemma \ref{PropertiesOfOpenAndClosed}, it is enough to prove that $U \subseteq V \to_{\mathcal{K}} U \cap V$, for any upsets $U, V \in B$. Let $x \in U$, $(x, y) \in R$ and $V \in N(y)$. Then, by the openness of the KN-frame, we know $U \cap V \in N(y)$, which completes the proof. A similar argument works for the closed case.
\end{example}

\begin{remark}\label{ThePassageToFull}
First, notice that a KN-frame is a combination of an intuitionistic Kripke frame with an independent monotone neighbourhood function restricted to the upsets of a given Boolean algebra of the subsets of $K$. The presence of the neighbourhood function is crucial as in the standard KN-frames, the definition of the implication simplifies to $U \to_{\mathcal{K}} V=\{x \in K \mid \forall y \in K \, [(x, y) \in R \; \text{and} \; y \in U \; \text{then} \; y \in V]\}$ which is always meet- and join-internalizing. Therefore, without the neighbourhood structure, the KN-frames are not capable of representing all implications. Secondly, note that starting from a KN-frame, it is always possible to drop the Boolean algebra $B$ to reach a full KN-frame and hence a greater strong algebra. More precisely, let $\mathcal{K}=(K, \leq, R, B, N)$ be a KN-frame and define $\mathcal{K}^f$ as $(K, \leq, R, P(K), N^f)$, where $N^f(x)=\{U \in U(K, \leq) \mid \exists V \in N(x) \, V \subseteq U\}$. It is easy to see that $\mathcal{K}^f$ is a full KN-frame. Moreover, as $N(x)$ is upward closed for the upsets in $B$, it is clear that $U \in N(x)$ iff $U \in N^f(x)$, for any $U \in U(K, \leq, B)$. Therefore, the strong algebra $
\mathfrak{A}(\mathcal{K})$ is a subalgebra of the strong algebra $
\mathfrak{A}(\mathcal{K}^f)$. Notice that the passage from $\mathcal{K}$ to $\mathcal{K}^f$ does not necessarily preserve the openness or the closedness of the original KN-frame. 
\end{remark}

\begin{example}\label{ExamOfSatndardFrame}
For any standard KN-frame $\mathcal{K}=(K, \leq, R)$, if $R \subseteq \, \leq$, the implication $\to_{\mathcal{K}}$ is open and if $R^{op} \subseteq \, \leq$, it is closed, where by $R^{op}$, we mean $\{(k, l) \in K^2 \mid (l, k) \in R\}$. For the first claim, using Lemma \ref{PropertiesOfOpenAndClosed}, it is enough to show $U \subseteq V \to_{\mathcal{K}} U \cap V$, for any $U, V \in U(K, \leq)$. For that purpose, assume $k \in U$. Then, for any $l \in V$, if $(k, l) \in R$, as $R \subseteq \, \leq$, we have $k \leq l$ and since $U$ is an upset, we have $l \in U$. Hence, $l \in U \cap V$. Therefore, $k \in V \to_{\mathcal{K}} (U \cap V)$. For the second claim, again using Lemma \ref{PropertiesOfOpenAndClosed}, it is enough to show $[(U \cup V) \to_{\mathcal{K}} V] \cup U=K$, for any $U, V \in U(K, \leq)$. Suppose $k \notin U$. Then, for any $l \in U \cup V$, if $(k, l) \in R$, as $R^{op} \subseteq \, \leq$, we have $l \leq k$. Hence, as $U$ is an upset and $k \notin U$, we have $l \notin U$. Hence, $l \in V$. Therefore, $k \in (U \cup V) \to_{\mathcal{K}} V$.\\
Employing these two families of implications, it is easy to provide a closed implication that is not open. Set $K=\{k, l\}$, $k \leq l$ and $R=\{(l, k), (k, k)\}$ and consider $\mathcal{K}=(K, \leq, R)$. It is clear that $R$ is compatible with the order and $R^{op} \subseteq \, \leq$. Hence, $\to_{\mathcal{K}}$ is closed. To show that it is not open, we show $K \to_{\mathcal{K}} \{l\}=\varnothing$ and as $\{l\} \nsubseteq K \to_{\mathcal{K}} \{l\}$, the implication cannot be open. For $K \to_{\mathcal{K}} \{l\}=\varnothing$, if either $k$ or $l$ is in $K \to_{\mathcal{K}} \{l\}$, as $(k,k), (l, k) \in R$, we must have $k \in \{l\}$ which is impossible.
\end{example}

\begin{remark}
It is not hard to see that being closed (as opposed to being open) is a very demanding condition restricting the form of the closed implications and hence the WBI's in a very serious way (see Theorem \ref{CharacOfWBI}). For instance, it is easy to see that the condition $R^{op} \subseteq \, \leq$ in Example \ref{ExamOfSatndardFrame} together with the compatibility condition of $R$ with respect to $\leq$, restricts the relation $R$ only to the ones in the from $\{(f(k), k) \in K^2 \mid k \in L\} \cup \{(k, k) \in K^2 \mid k \in L\}$, where $f: L \to K$ is an injective function, $L$ is a subset of the minimal elements of $K$ and $f(k) \geq k$, for any $k \in L$:

\begin{center} \small
\begin{tikzcd}[ampersand replacement=\&]
	{k_0} \& {k_1} \& {k_2} \\
	\&\&\& \cdots \\
	{f(k_{0})} \& {f(k_1)} \& {f(k_2)}
	\arrow["R", no head, from=3-1, to=1-1]
	\arrow["R", no head, from=3-2, to=1-2]
	\arrow["R", no head, from=3-3, to=1-3]
\end{tikzcd}
\end{center}
For the WBI's, even the function $f$ collapses to the identity and the only remaining data will be the set $L$. We do not prove these claims as they will not be used in the present paper. They are mentioned here only to convey the feeling that the study of these two families is not as justified as one might expect. However, we spend some time studying these two families as we need their behavior to prove the rarity of geometric implications in the next section. 
\end{remark}

So far, we have seen some concrete examples of the open and closed implications. The following theorem modifies the method of Example \ref{ExampleOfImplications} to construct the new open and closed implications from the old.

\begin{theorem}\label{OpenFromOpen}
Let $(\mathcal{B}, \to_{\mathcal{B}})$ be a strong algebra, $f: \mathcal{A} \to \mathcal{B}$ be an order-preserving and $g: \mathcal{B} \to \mathcal{A}$ be a finite meet preserving map. Then, $a \to_{\mathcal{A}} b=g(f(a) \to_{\mathcal{B}} f(b))$ is:
\begin{description}
\item[$(i)$]
an open implication over $\mathcal{A}$, if $\to_{\mathcal{B}}$ is open, $f$ preserves all binary meets and $gf(a) \geq a$, for any $a \in \mathcal{A}$.
\item[$(ii)$]
a closed implication over $\mathcal{A}$, if $\to_{\mathcal{B}}$ is closed, $f$ preserves all binary joins and $c \vee f(a)=1$ implies $g(c) \vee a=1$, for any $a \in \mathcal{A}$ and $c \in \mathcal{B}$.
\end{description}
\end{theorem}
\begin{proof}
To check whether $\to_{\mathcal{A}}$ is open or closed, we use the criterion of Lemma \ref{PropertiesOfOpenAndClosed}. For $(i)$, as $\to_{\mathcal{B}}$ is open, $f$ preserves the binary meets and $gf(a) \geq a$, we have $b \to_{\mathcal{A}} a \wedge b=g(f(b) \to_{\mathcal{B}} f(a \wedge b))=g(f(b) \to_{\mathcal{B}} f(a) \wedge f(b)) \geq g(f(a)) \geq a$. For $(ii)$, as $\to_{\mathcal{B}}$ is closed and $f$ preserves the binary joins, we have $[f(a \vee b) \to_{\mathcal{B}} f(b)] \vee f(a)=[f(a) \vee f(b) \to_{\mathcal{B}} f(b)] \vee f(a)=1$. Hence, by the property, we have $g(f(a \vee b) \to_{\mathcal{B}} f(b)) \vee a=1$ which implies $[(a \vee b) \to_{\mathcal{A}} b] \vee a=1$.
\end{proof}

\begin{remark}
In practice, part $(i)$ in Lemma \ref{OpenFromOpen} is useful in situations where $f \dashv g$ and $f$ is binary meet preserving and part $(ii)$ is used when $g \dashv f$, $g$ is finite meet preserving and $f$ is binary join preserving.  
\end{remark}

The following theorem provides a complete characterization of the WBI's. We will see that their general form is not far from the one provided in Example \ref{BasicExample}. 
\begin{theorem}\label{CharacOfWBI}
Let $\mathcal{A}$ be a bounded distributive lattice. Then, for any weakly Boolean implication $\to$ over $\mathcal{A}$, the interval $[\neg 1, 1]=\{x \in \mathcal{A} \mid \neg 1 \leq x \leq 1\}$ with its induced order is a Boolean algebra. Moreover, $a \to b=\neg a \vee b$ and $\neg a$ is the complement of $a \vee \neg 1$ in $[\neg 1, 1]$. Conversely, if for some $m \in \mathcal{A}$, the interval $[m, 1]$ with its induced order is a Boolean algebra, and $n(a)$ is the complement of $a \vee m$ in $[m, 1]$, then $a \to_m b=n(a) \vee b$ is a WBI over $\mathcal{A}$. Note that $\neg_m a=n(a)$ and $\neg_m 1=n(1)=m$.
\end{theorem}
\begin{proof}
For the first part, by Lemma \ref{PropertiesOfOpenAndClosed}, we have $(a \vee \neg 1) \wedge \neg a= (a \wedge \neg a)\vee (\neg 1 \wedge \neg a) =(a \wedge \neg 1) \vee \neg 1=\neg 1$ and $(a \vee \neg 1) \vee \neg a=1$, for any $a \in \mathcal{A}$. Therefore, $\neg a$ is the complement of $a \vee \neg 1$ over $[\neg 1, 1]$. Moreover, for any $a \geq \neg 1$, as $a \vee \neg 1=a$, the element $\neg a$ is the complement of $a$ which implies that $[\neg 1, 1]$ is a Boolean algebra.
To show $a \to b=\neg a \vee b$, note that $\neg a=(a \to 0) \leq (a \to b)$. As the implication is open, by Lemma \ref{PropertiesOfOpenAndClosed}, we have $b \leq a \to b$. Hence, $\neg a \vee b \leq a \to b$. For the converse, as the implication is also closed, by Lemma \ref{PropertiesOfOpenAndClosed}, $b \vee \neg b=1$. Hence, $a \to b=(a \to b) \wedge (b \vee \neg b)=((a \to b) \wedge b) \vee ((a \to b) \wedge \neg b) \leq b \vee (a \to 0)=\neg a \vee b$. \\
Conversely, let $[m, 1]$ be a Boolean algebra and $n(a)$ be the complement of $a \vee m$ in $[m, 1]$. First, we prove that $a \vee n(a)=1$ and $b \wedge n(b) \leq n(a)$, for any $a, b \in \mathcal{A}$. For the former, as $n(a)$ is the complement of $a \vee m$, we have $n(a) \vee (a \vee m)=1$ and as $n(a) \in [m, 1]$, we reach $m \leq n(a)$. Hence, $n(a) \vee a=1$. For the latter, as $n(b)$ is the complement of $b \vee m$ over $[m, 1]$, we have $(b \vee m) \wedge n(b)=m$ which implies $b \wedge n(b) \leq m$. Again as $m \leq n(a)$, we reach $b \wedge n(b) \leq n(a)$.
Now, to show that $a \to_m b=n(a) \vee b$ is an implication, we must check the properties in Remark \ref{OtherDefImp}. For $(i')$, if $a \leq b$, then $1=n(a) \vee a \leq n(a) \vee b=a \to_m b$. For $(ii)$, using the distributivity, we have
$
(a \to_m b) \wedge (b \to_m c)=(n(a) \vee b) \wedge (n(b) \vee c) \leq n(a) \vee (b \wedge n(b)) \vee c \leq n(a) \vee n(a) \vee c=a \to_m c$. Finally, to show that $\to_m$ is both open and closed, as $a \to_m (a \wedge b)=n(a) \vee (a \wedge b)=(n(a) \vee a) \wedge (n(a) \vee b)=(n(a) \vee b) \geq b$, the implication is clearly open. For closedness, note that $((a \vee b) \to_m a) \vee b=n(a \vee b) \vee a \vee b=1$.
\end{proof}

\begin{remark}\label{PropertiesOfWBI}
Note that in a Boolean algebra, the complement of $a \vee m$ in $[m, 1]$ is $\bar{a} \vee m$, where $\bar{a}$ is the complement of $a$. Therefore, by Theorem \ref{CharacOfWBI}, it is clear that the implications $a \to_m b=\bar{a} \vee b \vee m$ are the only WBI's over a Boolean algebra. Having made this observation, one wonders whether the presence of a WBI forces the ground lattice to be a Boolean algebra itself. However,  as the trivial implication is weakly Boolean and it is definable over any bounded distributive lattice, it is clear that the ground lattice structure of a weakly Boolean strong algebra can be quite general and not necessarily Boolean.
\end{remark}

\subsection{Weakly Boolean Spaces}\label{SubsectionWBIOverSpaces}

In this subsection, we will use Theorem \ref{CharacOfWBI} to provide a characterization for the WBI's over the locales of opens of the topological spaces. The characterization will be useful in the next section. First, let us recall some basic notions from topology that we need below.
A topological space $X$ is called \emph{discrete} if all of its subsets are open. It is called \emph{indiscrete} if its only opens are $\varnothing$ and $X$. It is called \emph{locally indiscrete} if for any $x \in X$, there exists an open $U \subseteq X$ such that $x \in U$ and the induced topology on $U$ is indiscrete. A space is locally indiscrete iff all of its closed subsets are open. More generally:

\begin{lemma}\label{Well-definedCore}
Let $X$ be a topological space and $M \subseteq X$ be an open subset. Then, the following are equivalent: 
\begin{description}
\item[$(i)$]
For any closed $K \subseteq X$, the union $K \cup M$ is open.
\item[$(ii)$]
The subspace $X-M$ is locally indiscrete.
\end{description}
\end{lemma}
\begin{proof}
To prove $(ii)$ from $(i)$, let $L$ be a closed subset of $X-M$. We have to show that $L$ is also open. As $L$ is closed in $X-M$, there is a closed subset $K$ of $X$ such that $L=K \cap (X-M)$. By $(i)$, the set $K \cup M$ is open. As $L=(K \cup M) \cap (X-M)$, the set $L$ is open in $X-M$. Conversely, to prove $(i)$ from $(ii)$, assume that $K$ is a closed subset in $X$. Hence, $K \cap (X-M)$ is closed and hence open in $X-M$, by $(ii)$. This implies the existence of an open $U$ in $X$ such that $K \cap (X-M)=U \cap (X-M) $. Hence, $K \cup M=U \cup M$. Finally, as both $U$ and $M$ are open in $X$, the subset $K \cup M$ is open in $X$.
\end{proof}
By definition, it is clear that any discrete space is locally indiscrete. The converse also holds for $T_0$ spaces. 

\begin{lemma}
Any locally indiscrete $T_0$ space is discrete.
\end{lemma}
\begin{proof}
Let $x \in X$ be a point. As $X$ is $T_0$, for any $y \neq x$, there is an open $U_y$ such that either $x \in U_y$ and $y \notin U_y$ or $x \notin U_y$ and $y \in U_y$. As any open is closed in the space, w.l.o.g, we can assume the first case. As closed and open subsets are identical, opens are closed under arbitrary intersections. Hence, set $U=\bigcap_{y \neq x} U_y$. It is easy to see that $U=\{x\}$. Hence, the singletons and consequently all subsets are open and hence the space is discrete.
\end{proof}

We are now ready to provide a characterization for all the WBI's over the locales of topological spaces, as promised.

\begin{corollary}\label{TheCor}
Let $X$ be a topological space and $M \subseteq X$ be an open subset such that $X-M$ is locally indiscrete. Then, the binary map $U \to V=U^c \cup V \cup M$ is a WBI over $\mathcal{O}(X)$. Conversely, any WBI over $\mathcal{O}(X)$ is in the form $U \to V=U^c \cup V \cup M$, where $M=\neg X $ is an open subset such that $X-M$ is locally indiscrete. 
\end{corollary}
\begin{proof}
For the first part, by Lemma \ref{Well-definedCore}, $\mathcal{O}(X)$ is closed under the operation $U \mapsto U^c \cup M$ as $X-M$ is locally indiscrete and hence the subset $U^c \cup M$ is open, for any open $U$. This proves that $\to$ is well-defined over $\mathcal{O}(X)$. Now, note that the interval $[M, X]$ in $\mathcal{O}(X)$ is a Boolean algebra, simply because $U^c \cup M$ is the complement of $U \supseteq M$. Therefore, by Theorem \ref{CharacOfWBI}, the operation $U \to V=n(U) \cup V$ is a WBI, where $n(U)$ is the complement of $U \cup M$ in $[M, X]$. As $n(U)=U \cup M^c$, we know that $U \to V=U^c \cup V \cup M$ is a WBI. Conversely, if $\to$ is a WBI over $\mathcal{O}(X)$, then by Theorem \ref{CharacOfWBI} again, the interval $[\neg X, X]$ is a Boolean algebra and $\neg U$ is the complement of $U \cup \neg X$ in $[\neg X, X]$. Set $M=\neg X$ and note that $M$ is open, as $\neg X \in \mathcal{O}(X)$.
As the complement of $U \cup M$ in $[M, X]$ is $U^c \cup M$, we have $U \to V=U^c \cup V \cup M$. Moreover, as $\neg U=U^c \cup M$ is open, for any open $U$, the space $X-M$ is locally indiscrete, by Lemma \ref{Well-definedCore}.
\end{proof}

\begin{remark}\label{TheRem}
For any WBI over $\mathcal{O}(X)$, its unique $M=\neg X$ is called its \emph{core}. Note that if $M=\varnothing$, then the space $X$ is locally indiscrete and the trivial and the Boolean implications are the WBI's with the cores $X$ and $\varnothing$, respectively. The strong space $(X, \to)$ is called a \emph{weakly Boolean space}, (WBS, for short), if $\to$ is a WBI over $\mathcal{O}(X)$. Note that if a space is locally indiscrete, then by Corollary \ref{TheCor}, for any open $N$, there is a WBI with the core $N$, because $X-N$ is locally indiscrete. Also notice that a continuous map $f: X \to Y$ induces a strong map $f: (X, \to_X) \to (Y, \to_Y)$ between two WBS's, iff $f^{-1}(M_Y)=M_X$.
\end{remark}

\begin{example}\label{ExamOfHaus}
Let $X$ be a topological space and $N \subseteq X$ be a closed discrete subspace. Then, by Corollary \ref{TheCor}, $U \to V=U^c \cup V \cup (X-N)$ is a WBI, as $X-N$ is open and $N$ is locally indiscrete. Corollary \ref{TheCor} also shows that if $X$ is $T_0$, then these are the only WBI's over $X$, as in a $T_0$ space, the only locally indiscrete subspaces are the discrete ones. Moreover, for the later reference, note that if $X$ is a Hausdorff space, then $U \to V=U^c \cup V \cup (X-N)$ is a WBI, for any finite $N \subseteq X$.
\end{example}

\subsection{Open-irreducible and Closed-irreducible Maps} \label{OpenClosedIrreducible}
In this subsection, we will first introduce the two families of open- and closed-irreducible maps as the generalizations of the injective open and closed continuous functions, respectively. Then, we will
prove that an implication $\to$ over $\mathcal{O}(X)$ is open (closed) iff the inverse image of any open-irreducible (closed-irreducible) map into $X$ transforms $\to$ to another implication. 

\begin{definition}
A topological space $X$ is \emph{open-irreducible (closed-irreducible)}, if $A \cap B=\varnothing$ implies $A=\varnothing$ or $B=\varnothing$, for any open (closed) subsets $A, B \subseteq X$. 
\end{definition}

Intuitively, the open-irreducible (closed-irreducible) spaces are roughly the spaces in which the open (closed) subsets are so big that any two non-empty  open (closed) subsets intersect. 

\begin{example}
Any singleton space is both open- and closed-irreducible. Any space that is the closure of a single point is open-irreducible because if $X=\overline{\{x\}}$ and $U$ and $V$ are two nonempty opens, then $x \in U$ as otherwise, $x \in U^c$ and since $U^c$ is closed, we have $U^c=\overline{\{x\}}$ which implies $U=\varnothing$. Similarly, $x \in V$. Therefore, $U \cap V \neq \varnothing$. For closed-irreducible spaces, it is easy to see that $X$ is closed-irreducible iff $\overline{\{x\}} \cap \overline{\{y\}}$ is non-empty, for any $x, y \in X$. For instance, if $(P, \leq)$ is a poset, where any two elements have an upper bound, then $P$ with the topology of the upsets as the closed subsets is closed-irreducible. The reason is that $\overline{\{x\}}$ is $\{z \in P \mid z \geq x\}$. Hence, if $w$ is the upper bound of $x$ and $y$, then $w \in \overline{\{x\}} \cap \overline{\{y\}}$.
\end{example}

For our purpose, we must look into the relative version of the open-irreducible (closed-irreducible) space, where a space is replaced by a continuous map.

\begin{definition}
A continuous map $f: X \to Y$ is called \emph{open-irreducible (closed-irreducible)} if it is open (closed) and $f[A] \cap f[B]=f[A \cap B]$, for any open (closed) subsets $A, B \subseteq X$. As the identity function is open-irreducible (closed-irreducible) and the composition of any two open-irreducible (closed-irreducible) maps is also open-irreducible (closed-irreducible), considering all topological spaces together with the open-irreducible (resp. closed-irreducible) maps forms a category that we denote by $\mathbf{OI}$ (resp. $\mathbf{CI}$).
\end{definition}

\begin{example}
The unique map $!: X \to \{*\}$ is open-irreducible (closed-irreducible) iff $X$ is open-irreducible (closed-irreducible). It is easy to prove that an open (a closed) map is open-irreducible (closed-irreducible) iff the fiber $f^{-1}(y)$ is open-irreducible (closed-irreducible) as a subspace of $X$, for any $y \in Y$. As a special case, all injective open (closed) maps are trivially open-irreducible (closed-irreducible), as their fibers are singletons or empty. 
\end{example}

Note that if $f: X \to Y$ is an open map, then the map $f_!: \mathcal{O}(X) \to \mathcal{O}(Y)$ defined by $f_!(U)=f[U]$ is the left adjoint of $f^{-1}$, as $U \subseteq f^{-1}(V)$ iff $f[U] \subseteq V$. The map $f_!$ is well-defined as $f$ is open. Having this left adjoint, we can say that for any open map $f$, it is open-irreducible iff the left adjoint of $f^{-1}$ preserves the binary meets. Similarly, if $f: X \to Y$ is a closed map, then the map $f_*: \mathcal{O}(X) \to \mathcal{O}(Y)$ defined by $f_*(U)=f[U^c]^c$ is the right adjoint of $f^{-1}$, as $f^{-1}(V) \subseteq U$ iff $V \subseteq f[U^c]^c$. Note that $f_*$ is well-defined as $f$ is closed. Again, using this right adjoint, we can say that for any closed map $f$, it is closed-irreducible iff the right adjoint of $f^{-1}$ preserves the binary joins.

In the following theorem, we will characterize the open (closed) implications based on their preservability under the inverse image of the open-irreducible (closed-irreducible) maps.

\begin{theorem}\label{PreservationUnderEmbedding}
For a strong space $(Y, \to_Y)$, the followings are equivalent:
\begin{description}
\item[$(i)$]
For any topological space $X$ and any open-irreducible map $f: X \to Y$, there is an open implication $\to_X$ over $\mathcal{O}(X)$ such that $f^{-1}(U \to_Y V)=f^{-1}(U) \to_X f^{-1}(V)$, for any $U, V \in \mathcal{O}(Y)$.
\item[$(ii)$]
For any topological space $X$ and any open-irreducible map $f: X \to Y$, there is an implication $\to_X$ over $\mathcal{O}(X)$ such that we have $f^{-1}(U \to_Y V)=f^{-1}(U) \to_X f^{-1}(V)$, for any $U, V \in \mathcal{O}(Y)$.
\item[$(iii)$]
For any topological space $X$ and any open embedding $f: X \to Y$, there is an implication $\to_X$ over $\mathcal{O}(X)$ such that $f^{-1}(U \to_Y V)=f^{-1}(U) \to_X f^{-1}(V)$, for any $U, V \in \mathcal{O}(Y)$.
\item[$(iv)$]
The implication $\to_Y$ is open localizable, i.e., for any open subset $Z \subseteq Y$ and any $U_1, V_1, U_2, V_2 \in \mathcal{O}(Y)$, if $U_1 \cap Z=U_2 \cap Z$ and $V_1 \cap Z=V_2 \cap Z$, then $(U_1 \to_Y V_1) \cap Z=(U_2 \to_Y V_2) \cap Z$.
\item[$(v)$]
The implication $\to_Y$ is open.
\end{description}
The same also holds, replacing open with closed, everywhere in the theorem.
\end{theorem}
\begin{proof}
The parts $(i)$ to $(ii)$ and $(ii)$ to $(iii)$ are trivial. For $(iii)$ to $(iv)$, as $Z$ is open in $Y$, the inclusion map $j: Z \to Y$ is an open embedding. Therefore, we can apply $(iii)$ to $j$. Then, as $j^{-1}(U_1)=U_1 \cap Z=U_2\cap Z=j^{-1}(U_2)$ and $j^{-1}(V_1)=V_1 \cap Z=V_2\cap Z=j^{-1}(V_2)$, we have $j^{-1}(U_1) \to_Z j^{-1}(V_1)=j^{-1}(U_2) \to_Z j^{-1}(V_2)$. Therefore, $j^{-1}(U_1 \to_Y V_1)=j^{-1}(U_2 \to_Y V_2)$ and hence, $(U_1 \to_Y V_1) \cap Z=(U_2 \to_Y V_2) \cap Z$. \\
For $(iv)$ to $(v)$, using Lemma \ref{PropertiesOfOpenAndClosed}, it is enough to prove $V \subseteq U \to_Y U \cap V$, for any $U, V \in \mathcal{O}(Y)$. Set $Z=V$ in $(iv)$. As $Z \cap (U \cap V)=Z \cap U$, we have $(U \to_Y U \cap V) \cap Z=(U \to_Y U) \cap Z$, by $(iv)$. As $\to_Y$ is an implication, we have $U \to_Y U=Y$ and hence $V=Z \subseteq (U \to_Y U \cap V)$.\\
Finally, to prove $(i)$ from $(v)$, define $U' \to_X V'=f^{-1}(f[U'] \to_Y f[V'])$, for any $U', V' \in \mathcal{O}(X)$. First, note that as $f$ is open, $\to_X$ is well-defined over $\mathcal{O}(X)$. Secondly, by Example \ref{ExampleOfImplications}, $\to_X$ is clearly an implication. Thirdly, recall that $f_!(U)=f[U]$ is the left adjoint of $f^{-1}$ and preserves all binary meets. Hence, as $\to_Y$ is open, by Theorem \ref{OpenFromOpen}, the implication $\to_X$ is open. \\
Finally, to prove $f^{-1}(U) \to_X f^{-1}(V)=f^{-1}(U \to_Y V)$, for any $U, V \in \mathcal{O}(Y)$, we first prove that $f^{-1}(U_1)=f^{-1}(U_2)$ and $f^{-1}(V_1)=f^{-1}(V_2)$ imply $f^{-1}(U_1 \to_Y V_1)=f^{-1}(U_2 \to_Y V_2)$, for any opens $U_1, V_1, U_2, V_2 \in \mathcal{O}(Y)$. First, notice that by the assumptions $f^{-1}(U_1)=f^{-1}(U_2)$ and $f^{-1}(V_1)=f^{-1}(V_2)$, we reach $U_1 \cap f[X]=U_2 \cap f[X]$ and $V_1 \cap f[X]=V_2 \cap f[X]$. As $f[X]$ is open, $\to_Y$ is an open implication, and $U_2 \cap f[X] \subseteq U_1$ and $V_1 \cap f[X] \subseteq V_2$, we have $f[X] \subseteq (U_2 \to_Y U_1) \cap (V_1 \to_Y V_2)$. Therefore, $f[X] \cap (U_1 \to_Y V_1) \subseteq (U_2 \to_Y U_1) \cap (U_1 \to_Y V_1) \cap (V_1 \to_Y V_2) \subseteq U_2 \to_Y V_2$. By symmetry, we also have $f[X] \cap (U_2 \to_Y V_2) \subseteq U_1 \to_Y V_1$. Therefore, $f[X] \cap (U_1 \to_Y V_1)=f[X] \cap (U_2 \to_Y V_2)$ which implies $f^{-1}(U_1 \to_Y V_1)=f^{-1}(U_2 \to_Y V_2)$.\\
Now, as $f_! \dashv f^{-1}$, we have $f^{-1}f_!f^{-1}(U)=f^{-1}(U)$ and $f^{-1}f_!f^{-1}(V)=f^{-1}(V)$, for any $U, V \in \mathcal{O}(Y)$. Therefore, using what we just proved, we have $f^{-1}(U) \to_X f^{-1}(V)=f^{-1}(f_!f^{-1}(U) \to_Y f_!f^{-1}(V))=f^{-1}(U \to_Y V)$.\\
For the closed case, the only non-trivial cases are $(iv)$ to $(v)$ and $(v)$ to $(i)$. For the first case, using Lemma \ref{PropertiesOfOpenAndClosed}, it is enough to show $(U \cup V \to_Y V) \cup U=Y$ or equivalently $U^c \subseteq U \cup V \to_Y V$, 
for any $U, V \in \mathcal{O}(Y)$. Set $Z=U^c$ in $(iv)$ and note that $Z$ is closed in $Y$. As $(U \cup V) \cap Z=V \cap Z$, we have $(U \cup V \to_Y V) \cap Z=(V \to_Y V) \cap Z$, by $(iv)$. As $\to_Y$ is an implication, we have $V \to_Y V=Y$ and hence $U^c=Z \subseteq (U \cup V \to_Y V)$. Therefore, $U^c \subseteq U \cup V \to_Y V$.\\
To prove $(i)$ from $(v)$, define $U' \to_X V'=f^{-1}(f_*(U') \to_Y f_*(V'))$, for any $U', V'\in \mathcal{O}(X)$, where $f_*(W)=(f[W^c])^c$, for any $W \in \mathcal{O}(X)$. By Example \ref{ExampleOfImplications}, $\to_X$ is clearly an implication. As $f^{-1} \dashv f_*$, the implication $\to_Y$ is closed and $f_*$ preserves all binary joins, by Theorem \ref{OpenFromOpen},  $\to_X$ is also closed. To prove $f^{-1}(U \to_Y V)=f^{-1}(U) \to_X f^{-1}(V)$, we first prove that $f^{-1}(U_1)=f^{-1}(U_2)$ and $f^{-1}(V_1)=f^{-1}(V_2)$ imply $f^{-1}(U_1 \to_Y V_1)=f^{-1}(U_2 \to_Y V_2)$, for any $U_1, V_1, U_2, V_2 \in \mathcal{O}(Y)$. As before, by $f^{-1}(U_1)=f^{-1}(U_2)$ and $f^{-1}(V_1)=f^{-1}(V_2)$, we reach $U_1 \cap f[X]=U_2 \cap f[X]$ and $V_1 \cap f[X]=V_2 \cap f[X]$. Therefore, $U_2 \subseteq f[X]^c \cup U_1$ and $V_1 \subseteq f[X]^c \cup V_2$. As $f[X]^c$ is open and the implication $\to_Y$ is closed, we have 
$
f[X]^c \cup (U_2 \to_Y U_1)=f[X]^c \cup (V_1 \to_Y V_2)=Y$. Hence, 
$
U_1 \to_Y V_1 \subseteq f[X]^c \cup (U_1 \to_Y V_1)$ which is itself a subset of
\[
(f[X]^c \cup (U_1 \to_Y V_1)) \cap (f[X]^c \cup (U_2 \to_Y U_1)) \cap (f[X]^c \cup (V_1 \to_Y V_2)])=
\]
\[
f[X]^c \cup [(U_2 \to_Y U_1) \cap (U_1 \to_Y V_1) \cap (V_1 \to_Y V_2)] \subseteq f[X]^c \cup (U_2 \to_Y V_2).
\]
Thus, $U_1 \to_Y V_1 \subseteq f[X]^c \cup (U_2 \to_Y V_2) $. By symmetry, $U_2 \to_Y V_2 \subseteq f[X]^c \cup (U_1 \to_Y V_1) $. Therefore, $f[X] \cap (U_1 \to_Y V_1)= f[X] \cap (U_2 \to_Y V_2) $ which implies $f^{-1}(U_1 \to_Y V_1)=f^{-1}(U_2 \to_Y V_2)$.\\
Finally, as $f^{-1} \dashv f_*$, we have the equalities $f^{-1}f_*f^{-1}(U)=f^{-1}(U)$ and $f^{-1}f_*f^{-1}(V)=f^{-1}(V)$, for any $U, V \in \mathcal{O}(Y)$. Therefore, 
\[
f^{-1}(U) \to_X f^{-1}(V)=f^{-1}(f_*f^{-1}(U) \to_Y f_*f^{-1}(V))
\]
which is equal to $f^{-1}(U \to_Y V)$,
for any $U, V \in \mathcal{O}(Y)$.
\end{proof}

\section{Geometric Implications}
In this section, we will introduce a natural notion of geometricity for a category of strong spaces over a category of spaces to formalize the informal concept of geometricity of a family of implications over a family of continuous maps. Then, we will provide a characterization theorem for the geometric categories over the categories of spaces satisfying some basic closure properties. Specifically, we will see that the only geometric category over $\mathbf{Top}$ is the category of the strong spaces with the trivial implication. 
We first need the following lemma.
\begin{lemma}\label{Lifting}
\begin{description}
\item[$(i)$]
Let $(\mathcal{A}, \to_{\mathcal{A}})$ be a strong algebra, $\mathcal{B}$ be a bounded distributive lattice and $f: \mathcal{A} \to \mathcal{B}$ be a bounded lattice map with a left inverse, i.e, a monotone function $g: \mathcal{B} \to \mathcal{A}$ such that $gf=id_{\mathcal{A}}$. Then, there is an implication $\to_{\mathcal{B}}$ over $\mathcal{B}$ such that $f: (\mathcal{A}, \to_{\mathcal{A}}) \to (\mathcal{B}, \to_{\mathcal{B}})$ is a strong algebra map.
\item[$(ii)$]
Let $(\mathcal{A}, \to_{\mathcal{A}})$ be a weakly Boolean strong algebra, $\mathcal{B}$ be a bounded distributive lattice and $f: \mathcal{A} \to \mathcal{B}$ be a surjective bounded distributive lattice morphism. Then, there is a unique weakly Boolean implication $\to_{\mathcal{B}}$ over $\mathcal{B}$ such that $f: (\mathcal{A}, \to_{\mathcal{A}}) \to (\mathcal{B}, \to_{\mathcal{B}})$ is a strong algebra map. If $\mathcal{B}$ is Boolean, it holds even without the surjectivity condition.
\end{description}
\end{lemma}
\begin{proof}
For $(i)$, define $c \to_{\mathcal{B}} d=f(g(c) \to_{\mathcal{A}} g(d))$, for any $c, d \in \mathcal{B}$, where $g: \mathcal{B} \to \mathcal{A}$ is the left inverse of $f$. By Example \ref{ExampleOfImplications}, the operator $\to_{\mathcal{B}}$ is an implication. To prove that $f$ is a strong algebra map, note that $f(a) \to_{\mathcal{B}} f(b)=f(gf(a) \to_{\mathcal{A}} gf(b))$. As $gf=id_{\mathcal{A}}$, we reach the conclusion.\\
For $(ii)$, as $\to_{\mathcal{A}}$ is weakly Boolean, by Theorem \ref{CharacOfWBI}, the interval $[\neg_{\mathcal{A}} 1, 1]$ is a Boolean algebra, $a \to_{\mathcal{A}} b=\neg_{\mathcal{A}} a \vee b$ and $\neg_{\mathcal{A}} a$ is the complement of $a \vee \neg_{\mathcal{A}} 1$ in $[\neg_{\mathcal{A}} 1, 1]$.
Set $n=f(\neg_{\mathcal{A}} 1)$. Note that $f(\neg_{\mathcal{A}} a)$ is the complement of $f(a) \vee n$ in $[n, 1_{\mathcal{B}}]$, for any $a \in \mathcal{A}$, because
$f(a) \vee n \vee f(\neg_{\mathcal{A}} a)=f(a \vee \neg_{\mathcal{A}} a \vee \neg_{\mathcal{A}} 1)=f(1_{\mathcal{A}})=1_{\mathcal{B}}$ and
$(f(a) \vee n) \wedge f(\neg_{\mathcal{A}} a)=f(a \vee \neg_{\mathcal{A}} 1) \wedge f(\neg_{\mathcal{A}} a)= f((a \vee \neg_{\mathcal{A}} 1) \wedge \neg_{\mathcal{A}} a)=f(\neg_{\mathcal{A}} 1)=n$. 
Now, to show that the interval $[n, 1_{\mathcal{B}}]$ in $\mathcal{B}$ is a Boolean algebra, we use either the surjectivity of $f$ or the Booleanness of $\mathcal{B}$. If $f$ is surjective, then every $c \in \mathcal{B}$ is equal to $f(a)$, for some $a \in \mathcal{A}$. Hence, if $c=f(a) \geq n$, as $f(\neg_{\mathcal{A}} a)$ is the complement of $f(a) \vee n=f(a)$ in $[n, 1_{\mathcal{B}}]$, then any $c \geq n$ has a complement in $[n, 1_{\mathcal{B}}]$. If $\mathcal{B}$ is Boolean, then the complement of $c \geq n$ in $[n, 1_{\mathcal{B}}]$ is $\bar{c} \vee n$, where $\bar{c}$ is the complement of $c$ in $
\mathcal{B}$. Now, by Theorem \ref{CharacOfWBI}, the operation 
$c \to_{\mathcal{B}} d=n(c) \vee d$ is a WBI over $\mathcal{B}$, where $n(c)$ is the complement of $c \vee n$ in $[n, 1_{\mathcal{B}}]$. 
Note that we proved $n(f(a))=f(\neg_{\mathcal{A}} a)$. Hence, as $\neg_{\mathcal{B}} f(a)=n(f(a))$, we reach $f(a \to_{\mathcal{A}} b)=f(\neg_{\mathcal{A}} a \vee b)=f(\neg_{\mathcal{A}} a) \vee b=\neg_{\mathcal{B}} f(a) \vee f(b)=f(a) \to_{\mathcal{B}} f(b)$.
For uniqueness, let $\to$ be another WBI over $\mathcal{B}$ such that $f(a \to_{\mathcal{A}} b)=f(a) \to f(b)$, for any $a, b \in \mathcal{A}$. Then, $\neg 1_{\mathcal{B}}=f(\neg_{\mathcal{A}}1)=n$. As $\neg c$ and $\neg_{\mathcal{B}} c$ are both the complements of $c \vee n$ in $[n, 1_{\mathcal{B}}]$, they must be equal. As a WBI is determined uniquely by its negation, we have  $\to_{\mathcal{B}}=\, \to$.
\end{proof}

As geometricity is the stability of a family of implications under the inverse image of a family of continuous functions, to formalize this notion, we need to be precise about two ingredients: the continuous maps we use and the family of implications we choose. For the former, it is reasonable to start with a subcategory $\mathcal{S}$ of $\mathbf{Top}$ to have a relative version of geometricity. For the latter, as any implication must be over a space in this case, a natural formalization of a family of implications is some sort of fibration that to each space $X$ in $\mathcal{S}$ assigns a fiber of strong spaces over $X$. Having these two ingredients fixed, the geometricity simply means the stability of the fibres under the inverse image of the maps in $\mathcal{S}$. In other words, it states that for any map $f: X \to Y$ in $\mathcal{S}$, the inverse image map $f^{-1}$ must map a fiber over $Y$ into the fiber over $X$. The following is the formalization of this idea.

\begin{definition}
Let $\mathcal{S}$ be a (not necessarily full) subcategory  of $\mathbf{Top}$. A category $\mathcal{C}$ of strong spaces is called \emph{geometric over $\mathcal{S}$}, if the forgetful functor $U: \mathcal{C} \to \mathbf{Top}$ mapping $(X, \to_X)$ to $X$ has the following conditions:
\begin{description}
\item[$(i)$]
$U$ maps $\mathcal{C}$ to $\mathcal{S}$ and it is surjective on the objects of $\mathcal{S}$, and
\item[$(ii)$]
for any object $(Y, \to_Y)$ in $\mathcal{C}$, any object $X$ in $\mathcal{S}$ and any map $f: X \to Y=U(Y, \to_Y)$ in $\mathcal{S}$, there exists an object $(X, \to_X)$ in $\mathcal{C}$ such that $f$ induces a strong space map $f: (X, \to_X) \to (Y, \to_Y)$ in $\mathcal{C}$:
\begin{center}
    
\small \begin{tikzcd}
	{\mathcal{C}} && {(X, \to_X)} && {(Y, \to_Y)} \\
	&& {} && {} \\
	&& {} && {} \\
	{\mathcal{S}} && X && {Y=U(Y, \to_Y)}
	\arrow["U"', from=1-1, to=4-1]
	\arrow["f"', from=4-3, to=4-5]
	\arrow["f", from=1-3, to=1-5]
	\arrow[squiggly, from=2-5, to=3-5]
	\arrow[squiggly, from=2-3, to=3-3]
\end{tikzcd}
\end{center}
\end{description}
\end{definition}

Let us explain the connection between the definition and the previous discussion we had. Using the functor $U$, the category $\mathcal{C}$ is nothing but a way to provide a fiber of strong spaces or equivalently a fiber of implications over any space in $\mathcal{S}$. The condition $(i)$, then, demands that the fibers and the maps between them are all lying over $\mathcal{S}$ and none of the fibers are empty. The condition $(ii)$ is the geometricity condition that states that for any map $f: X \to Y$, its inverse image can pull back any implication over $Y$ to an implication over $X$. In other words, the fibers are stable under the inverse image of the maps of $\mathcal{S}$.\\ 
Note that if $\mathcal{C}$ is geometric over $\mathcal{S}$, then the maps of $\mathcal{C}$ and $\mathcal{S}$ are the same. One direction is easy as $U$ maps $\mathcal{C}$ into $\mathcal{S}$. For the other direction, notice that for any map $f: X \to Y$, as the fiber over $Y$ is non-empty, there exists an object $(Y, \to_Y)$ in $\mathcal{C}$. Now, use $(ii)$ to show that $f$ actually lives in $\mathcal{C}$. This implies that to show the equality of two geometric categories over a fixed category, it is enough to show that their objects are the same.
\begin{example}\label{ExamplesOfGeometricCats}
For any category $\mathcal{S}$ of spaces, let $\mathcal{S}_t$ be the category of strong spaces $(X, \to_t)$, where $X$ is in $\mathcal{S}$ and $\to_t$ is the trivial implication together with all the maps of $\mathcal{S}$ as the morphisms. It is well defined as the inverse image of any map $f: X \to Y$ in $\mathcal{S}$ preserves the implication, since $f^{-1}(Y)=X$. It is clear that $\mathcal{S}_t$ is a geometric category over $\mathcal{S}$. This category is called the \emph{trivial} geometric category over $\mathcal{S}$.

If $\mathcal{S}$ only consists of locally indiscrete spaces, then there are three other degenerate geometric categories over $\mathcal{S}$. The first is the category $\mathcal{S}_b$ of strong spaces $(X, \to_b)$, where $X$ is in $\mathcal{S}$ and $\to_b$ is the Boolean implication together with the maps of $\mathcal{S}$ as the morphisms. This category is well defined, since the locally indiscreteness of $X$ implies the Booleanness of $\mathcal{O}(X)$ and the inverse images always preserve all the Boolean operators. It is easy to see that $\mathcal{S}_b$ is actually geometric over $\mathcal{S}$. The second example is the union of $\mathcal{S}_b$ and $\mathcal{S}_t$ that we denote by $\mathcal{S}_{bt}$. This category is also clearly geometric over $\mathcal{S}$. The third example is $\mathcal{S}_a$, the subcategory of strong spaces $(X, \to)$, where $X$ is in $\mathcal{S}$ and $\to$ is a WBI, together with the strong space morphisms that $U$ maps into $\mathcal{S}$. The fibers are clearly non-empty, because of the existence of the trivial implication over any space. It is geometric by part $(ii)$ of Lemma \ref{Lifting}.

To have a non-degenerate example, note that the category of all strong spaces $(X, \to)$, where $\to$ is open (closed) together with the strong maps that are open-irreducible (closed-irreducible) is geometric over the category $\mathbf{OI}$ (resp. $\mathbf{CI}$), by Theorem \ref{PreservationUnderEmbedding}. Note that Theorem \ref{PreservationUnderEmbedding} actually proves that these categories are the greatest geometric categories over $\mathbf{OI}$ and $\mathbf{CI}$, respectively, because if the inverse image of any open-irreducible (closed-irreducible) map into $X$ maps an implication $\to$ over $\mathcal{O}(X)$ to another implication, then $\to$ is open (resp. closed). Finally, it is worth mentioning that the category of all strong spaces with the strong space morphisms that are also surjective is geometric over the category of all spaces with all surjections, $\mathbf{Srj}$. To prove, note that if $f: X \to Y$ is a continuous surjection, then $f^{-1}$ has a left inverse that is also monotone. Now, use part $(i)$ of Lemma \ref{Lifting}.
\end{example}
To provide a characterization for all geometric categories over a given category $\mathcal{S}$, it is reasonable to assume some closure properties for  $\mathcal{S}$.
\begin{definition}
A subcategory $\mathcal{S}$ of $\mathbf{Top}$ is called \emph{local} if it has at least one non-empty object and it is closed under all embeddings, i.e., for any space $X$ in $\mathcal{S}$ and any embedding $f: Y \to X$, both $Y$ and $f$ belongs to $\mathcal{S}$. A space $X$ is called \emph{full} in $\mathcal{S}$ if $\mathcal{S}$ has $X$ as an object and all maps $f: Y \to X$ as maps, for any object $Y$ in $\mathcal{S}$. 
\end{definition}

First, note that a local subcategory $\mathcal{S}$ of $\mathbf{Top}$ has all singleton spaces as its objects. Because, as $\mathcal{S}$ has a non-empty space $X$, it is possible to pick $x \in X$ and set $f: \{*\} \to X$ as a function mapping $*$ to $x$. As $f$ is an embedding and $\mathcal{S}$ is local, the space $\{*\}$ is in $\mathcal{S}$. Secondly, notice that although the singletons are all present in any local subcategory $\mathcal{S}$ of $\mathbf{Top}$, it does not necessarily mean that $\mathcal{S}$ has a terminal object. To have a counter-example, it is enough to set $\mathcal{S}$ as the subcategory of all embeddings in $\mathbf{Top}$. $\mathcal{S}$ is trivially local. However, if $T$ is its terminal object, then it must be a singleton, as otherwise, there are either two or none embeddings from $\{*\}$ to $T$. However, if $T$ is a singleton, then there is no embedding from the discrete space $\{*, \dagger\}$ with two points into $T$ which is a contradiction. \\
Using an argument similar to what we had above, it is not hard to see that a terminal object in a local subcategory of $\mathbf{Top}$, if exists, must be a singleton space.
Moreover, a singleton space is a terminal object iff it is full in the subcategory.

\begin{theorem}\label{MainCharacterization}
Let $\mathcal{S}$ be a local subcategory of $\mathbf{Top}$ with a terminal object:
\begin{description}
\item[$(i)$]
If $\mathcal{S}$ has at least one non-locally-indiscrete space, then the only geometric category over $\mathcal{S}$ is $\mathcal{S}_t$.
\item[$(ii)$]
If $\mathcal{S}$ only consists of locally-indiscrete spaces, includes a non-indiscrete space and a full discrete space with two points, then the only geometric categories over $\mathcal{S}$ are the four distinct categories $\mathcal{S}_t$, $\mathcal{S}_b$, $\mathcal{S}_{bt}$ and $\mathcal{S}_a$. 
\item[$(iii)$]
If $\mathcal{S}$ only consists of indiscrete spaces, then the only geometric categories over $\mathcal{S}$ are the three distinct categories $\mathcal{S}_t$, $\mathcal{S}_b$, and $\mathcal{S}_{bt}$.
\end{description}
\end{theorem}
\begin{proof}
We first show that the locality of $\mathcal{S}$ implies that the strong spaces in $\mathcal{C}$ are all WBS. Let $(Y, \to_Y)$ be an object in $\mathcal{C}$. As $\mathcal{S}$ is closed under all embeddings, by geometricity of $\mathcal{C}$ over $\mathcal{S}$, for any embedding $f: X \to Y$, there is a strong space $(X, \to_X)$ such that $f$ induces a strong space map from $(X, \to_X)$ to $(Y, \to_Y)$. Therefore, by Theorem \ref{PreservationUnderEmbedding}, $\to_Y$ is both open and closed and hence a WBI. Now, note that by Corollary \ref{TheCor}, $\to_Y$ has a core. Denote this core by $M_Y$. Then, the geometricity implies that for any $f: X \to Y$ in $\mathcal{S}$, there exists a strong space $(X, \to_X)$ in $\mathcal{C}$ such that $f$ induces the strong space map $f:(X, \to_X) \to (Y, \to_Y)$. Therefore, by Remark \ref{TheRem}, we can conclude that for any $f: X \to Y$ in $\mathcal{S}$, there is a WBI over $X$ in $\mathcal{C}$ with the core $M_X=f^{-1}(M_Y)$. We will use this property frequently in the proof.\\
To prove $(i)$, it is enough to show that $M_Y=Y$, for any $(Y, \to_Y)$ in $\mathcal{C}$. Because, by Remark \ref{TheRem}, any strong space over $Y$ must have the trivial implication and as the fiber is non-empty, it also has the trivial strong space. Thus, the objects of $\mathcal{C}$ are the same as the objects of $\mathcal{S}_t$ and hence $\mathcal{C}=\mathcal{S}_t$. Now, let $X$ be a non-locally-indiscrete space in $\mathcal{S}$ and for the sake of contradiction, assume $M_Y \neq Y$, for some $(Y, \to_Y)$ in $\mathcal{C}$. Set $y \in Y-M_Y$. As $\mathcal{S}$ is local and has a terminal object, the singleton space $\{*\}$ is a full object in $\mathcal{S}$. As $\mathcal{S}$ is closed under the embeddings, $\mathcal{S}$ has $Y-M_Y$ as an object and the embeddings $i: Y-M_Y \to Y$ and $k: \{*\} \to Y-M_Y$ as the morphisms, where $i$ is the inclusion and $k$ maps $*$ to $y$. Consider the map $f=ik!: X \to Y$, where $!: X \to \{*\}$ is the constant function. As $\{*\}$ is full, $!$ and hence $f$ lives in $\mathcal{S}$. By the geometricity of $\mathcal{C}$ over $\mathcal{S}$, there is an implication $\to_X$ over $X$ such that $f^{-1}(M_Y)=M_X$. As $f$ is a constant function mapping every element of $X$ to $y \notin M_Y$, we have $f^{-1}(M_Y)=\varnothing$. Hence, $M_X=\varnothing$. Therefore, by Remark \ref{TheRem}, $X$ is locally indiscrete which is a contradiction. Therefore, $M_Y=Y$. \\
For $(ii)$, it is easy to see that the four geometric categories over $\mathcal{S}$ are different.
Let $X$ be a non-indiscrete space in $\mathcal{S}$. As it is not indiscrete, there is a non-trivial open $\varnothing \subset M \subset X$. As $X$ is locally indiscrete, by Remark \ref{TheRem}, there is a WBI over $X$ with the core $M$. This WBI is clearly different from the Boolean and the trivial ones, as its core, i.e., $M$ is different from the cores of the other two that are $\varnothing$ and $X$. Hence, the fibers over $X$ in the four categories $\mathcal{S}_t$, $\mathcal{S}_b$, $\mathcal{S}_{bt}$ and $\mathcal{S}_a$ are different.
To prove that these four categories are the only geometric categories over $\mathcal{S}$, we will consider the following two cases. Either, $M_Y=\varnothing$ or $M_Y=Y$, for any $(Y, \to_Y)$ in $\mathcal{C}$ or there exists a strong space $(Y, \to_Y)$ in $\mathcal{C}$ such that $\varnothing \subset M_Y \subset Y$. In the first case, we show that $\mathcal{C}$ is uniquely determined by the fiber over $\{*\}$ and it is one of $\mathcal{S}_t$, $\mathcal{S}_b$ or $\mathcal{S}_{bt}$. In the second case, we show that $\mathcal{C}=\mathcal{S}_a$.
For the first case, as $\{*\}$ has only two subsets, there are only two possible implications over $\{*\}$. If both are in $\mathcal{C}$, then as the constant function $!: Y \to \{*\}$ lives in $\mathcal{S}$, by the fullness of $\{*\}$, the geometricity dictates the existence of two strong spaces $(Y, \to_Y)$ and $(Y, \to'_Y)$ in $\mathcal{C}$ such that $!^{-1}(\{*\})=M_Y$ and $!^{-1}(\varnothing)=M'_Y$. Therefore, $M_Y=Y$ and $M'_Y=\varnothing$. Hence, over $Y$ the implications with both cores appear. Hence, on objects, $\mathcal{C}$ and $\mathcal{S}_{bt}$ are similar which implies $\mathcal{C}=\mathcal{S}_{bt}$. For the other case, if the fiber over $\{*\}$ has just one implication, denote its core by $M_*$. Now, pick an arbitrary non-empty strong space $(Y, \to_Y)$ in $\mathcal{C}$ with the core $M_Y$. Note that it is either $\varnothing$ or $Y$, by the assumption. As $Y$ is non-empty, there is a map $i: \{*\} \to Y$. This is an embedding and hence it is in $\mathcal{S}$. Therefore, by geometricity, $i^{-1}(M_Y)=M_*$. Hence, if $M_*=\varnothing$, we have $M_Y=\varnothing$ and if $M_*=\{*\}$, we reach $M_Y=Y$. Therefore, by the same reasoning as before, either $\mathcal{C}=\mathcal{S}_b$ or $\mathcal{C}=\mathcal{S}_t$.\\
For the second case, assume the existence of a strong space $(Y, \to_Y)$ in $\mathcal{C}$ such that $\varnothing \subset M_Y \subset Y$. For any $Z$ in $\mathcal{S}$ and any open $N \subseteq Z$, define $\to_N$ as the WBI with the core $N$. As $Y$ is locally indiscrete, this is possible by Remark \ref{TheRem}. We show that $(Z, \to_N)$ is in $\mathcal{C}$ and as $N$ is arbitrary, we conclude $\mathcal{C}=\mathcal{S}_a$. Let $\{*, \dagger\}$ be the full discrete space with two elements in $\mathcal{S}$. Set $y \in M_Y$ and $z \in Y-M_Y$ and let $f: Z \to \{*, \dagger\}$ be the function mapping $N$ to $*$ and $Z-N$ to $\dagger$. As $N$ is open and $Z$ is locally indiscrete, $N$ is also clopen. Hence, $f$ is continuous. As $\mathcal{S}$ has all maps into $\{*, \dagger\}$, the map $f$ is in $\mathcal{S}$. Now, consider the map $i: \{*, \dagger\} \to Y$ mapping $*$ to $y$ and $\dagger$ to $z$. As $\{*, \dagger\}$ is discrete, the map $i$ is continuous. As $M_Y$ is open and $Y$ is locally indiscrete, $M_Y$ is clopen.
Hence, it is easy to see that $i$ is an embedding. Therefore, $i$ is in $\mathcal{S}$. Consider the map $g: Z \to Y$ defined by $g=if$ and notice that $g$ is also in $\mathcal{S}$. As $\mathcal{C}$ is geometric over $\mathcal{S}$, there must be a strong space $(Z, \to_Z)$ with the core $M_Z$ such that $g^{-1}(M_Y)=M_Z$. By construction, we have $g^{-1}(M_Y)=f^{-1}(i^{-1}(M_Y))=f^{-1}(\{*\})=N$. Therefore, $M_Z=N$. As any WBI is uniquely determined by its core, we have $\to_Z=\to_N$. Hence, $(Z, \to_N)$ is in $\mathcal{C}$.\\
For $(iii)$, note that the trivial and the Boolean implications are different over $\mathcal{O}(\{*\})$, as their cores, i.e., $\{*\}$ and $\{*\}^c=\varnothing$ are different. Hence, the three categories $\mathcal{S}_t$, $\mathcal{S}_b$ and $\mathcal{S}_{bt}$ are different, as their fibers over $\{*\}$ are different. To prove that these are the only geometric categories over $\mathcal{S}$, assume that $(X, \to_X)$ is an arbitrary object in $\mathcal{C}$. As $X$ is indiscrete and has only two opens, the core is either $\varnothing$ or $X$. The rest is similar to the first case in $(ii)$. 
\end{proof}

\begin{corollary}\label{MainCor}
Let $\mathcal{S}$ be a full subcategory of $\mathbf{Top}$ that is also local. Then: 
\begin{description}
\item[$(i)$]
If $\mathcal{S}$ has at least one non-locally-indiscrete space, then the only geometric category over $\mathcal{S}$ is $\mathcal{S}_t$.
\item[$(ii)$]
If $\mathcal{S}$ only consists of locally-indiscrete spaces and includes a non-indiscrete space, then the only geometric categories over $\mathcal{S}$ are the four distinct categories $\mathcal{S}_t$, $\mathcal{S}_b$, $\mathcal{S}_{bt}$ and $\mathcal{S}_a$. 
\item[$(iii)$]
If $\mathcal{S}$ only consists of indiscrete spaces, then the only geometric categories over $\mathcal{S}$ are the three distinct categories $\mathcal{S}_t$, $\mathcal{S}_b$, and $\mathcal{S}_{bt}$.
\end{description}
\end{corollary}
\begin{proof} 
As $\mathcal{S}$ is local, it has the singleton object and as $\mathcal{S}$ is a full subcategory, the singleton is a full object and hence it is the terminal object of $\mathcal{S}$. Having made that observation, for $(i)$ and $(iii)$, it is enough to use Theorem \ref{MainCharacterization}. For $(ii)$, let $Y$ be a non-indiscrete yet locally indiscrete space in $\mathcal{S}$. Hence, it has a non-trivial open $\varnothing \subset U \subset Y$ and as $Y$ is locally indiscrete, $U$ is clopen. Pick two elements $x \in U$ and $y \in Y-U$. The map $g: \{*, \dagger\} \to X$ mapping $*$ to $x$ and $\dagger$ to $y$ is continuous as $\{*, \dagger\}$ is discrete. It is an embedding, as $U$ is clopen. Hence, as $\mathcal{S}$ is local, it must have the discrete space $\{*, \dagger\}$. Finally, as $\mathcal{S}$ is a full subcategory, the space $\{*, \dagger\}$ is a full object. Now, use Theorem \ref{MainCharacterization} to complete the proof.
\end{proof}

\begin{corollary}
$\mathbf{Top}_t$ is the only geometric category over $\mathbf{Top}$.
\end{corollary}

The following example shows that the existence of the terminal object in Theorem \ref{MainCharacterization} is crucial, as without the condition, it is possible to have infinitely many geometric categories over some local subcategories of $\mathbf{Top}$.

\begin{theorem}
Let $\mathcal{S}$ be a category of Hausdorff spaces with injective continuous maps such that the size of the objects is not bounded by any finite number. Then, there are infinitely many geometric categories over $\mathcal{S}$.
\end{theorem}
\begin{proof}
For any natural number $n \geq 0$, define the category $\mathcal{C}_n$ as follows. For objects, consider all strong spaces $(X, \to_M)$, where $X$ is in $\mathcal{S}$, $M \subseteq X$ is a subset such that $X-M$ has at most $n$ elements and $\to_M$ is the WBI with the core $M$. Note that as $X-M$ is finite and $X$ is Hausdorff, $X-M$ is closed and discrete.  Therefore, the implication is well-defined by Example \ref{ExamOfHaus}. For the morphisms, consider all maps $f: (X, \to_M) \to (Y, \to_N)$, where $f: X \to Y$ is in $\mathcal{S}$ and $f^{-1}(N)=M$. By Remark \ref{TheRem}, the maps are strong space maps and hence well defined. We prove that $\mathcal{C}_n$ is geometric over $\mathcal{S}$. It is clear that $U$ maps $\mathcal{C}_n$ into $\mathcal{S}$. The fibers are non-empty as the trivial space $(X, \to_t)$ is in $\mathcal{C}_n$. For geometricity, for any $(Y, \to_N)$ and any map $f: X \to Y$ in $\mathcal{S}$, as $Y-N$ has at most $n$ elements and $f$ is injective, the set $f^{-1}(Y-N)=X-f^{-1}(N)$ has at most $n$ elements, as well. Hence, $f$ lifts to the strong space map $f:(X, \to_{f^{-1}(N)}) \to (Y, \to_N)$. Finally, it is clear that $\mathcal{C}_n$ is a subcategory of $\mathcal{C}_{n+1}$. We show that this inclusion is proper. To prove it, note that the size of the objects in $\mathcal{S}$ is not bounded by a finite number. Hence, for any $n$, there exists a space $X$ with at least $n+1$ elements. Choose $M$ such that $X-M$ has exactly $n+1$ elements. Then, $(X, \to_M)$ appears in $\mathcal{C}_{n+1}$ but not in $\mathcal{C}_n$. 
\end{proof}

\section{A Kripke-style Representation Theorem}\label{SecRepresentation}
In the previous sections, we have introduced three families of implications and their corresponding geometric characterizations. These implications are also interesting in their own right for their algebraic and logical aspects. Algebraically, they are different internalizations of the order of the base lattice, while logically, as observed in Remark \ref{LogicalJustification}, they are reflecting different versions of the right implication rule. Among the three, the closed and the weakly Boolean implications are rather too restricted in their form and hence too close to the well-understood Boolean implication. However, the open implications may come in many different forms, appearing in many different logical disciplines. For instance, the open meet-internalizing and join-internalizing implications are well-studied in provability logics \cite{visser1981propositional} and philosophical discussions around the impredicativity of the intuitionistic implication \cite{Ru}. Inspired by these aspects, in this section, we intend to provide a representation theorem for all implications, in general, and the open implications, in particular.

To provide this representation theorem, we first need to use a Yoneda-type technique augmented by an ideal completion to embed any strong algebra in a strong algebra over a locale, where the implication becomes a part of an adjunction-like situation. The machinery is explained in detail in \cite{akbar2021implication}. However, for the sake of completeness and to address the new case of the open implications, we will repeat the main construction here and refer the reader to \cite{akbar2021implication}, for a more detailed explanation.

\begin{lemma}\label{RepresentationLemma}
For any strong algebra $(\mathcal{A}, \to_{\mathcal{A}})$, there exist a locale $\mathcal{H}$, an implication $\to_{\mathcal{H}}$ over $\mathcal{H}$, an embedding $i : (\mathcal{A}, \to_{\mathcal{A}}) \to (\mathcal{H}, \to_{\mathcal{H}})$, a join preserving operator $\nabla: \mathcal{H} \to \mathcal{H}$ and an order-preserving map $F: \mathcal{H} \to \mathcal{H}$ such that $\nabla x \wedge F(y) \leq F(z)$ iff $x \leq y \to_{\mathcal{H}} z$, for any $x, y, z \in \mathcal{H}$. If $\to_\mathcal{A}$ is open, then $\to_\mathcal{H}$ can be chosen as open.
\end{lemma}
\begin{proof}
Let $\mathcal{B}$ be the set of all functions $f: \prod_{n=0}^{\infty} \mathcal{A} \to \mathcal{A}$ with the pointwise ordering $\leq_{\mathcal{B}}$. It is clear that $(\mathcal{B}, \leq_{\mathcal{B}})$ is a bounded distributive lattice. It is also possible to define an implication $\to_{\mathcal{B}}$ over $\mathcal{B}$ in a pointwise fashion. Define $s: \mathcal{B} \to \mathcal{B}$ as a shift function on the inputs, i.e., $[s(f)](\{x_i\}_{i=0}^{\infty})=f(\{x_i\}_{i=1}^{\infty})$. Notice that $s$ respects the order $\leq_{\mathcal{B}}$, the meets and the joins as they are defined pointwise. Set $\mathcal{H}$ as the locale of the ideals of $
\mathcal{B}$ and define $i: \mathcal{A} \to \mathcal{H}$ as $i(a)=\{f \in \mathcal{B} \mid f \leq c_a\}$, where $c_a$ is the constant function mapping every input into $a$. It is clear that $i$ is an embedding respecting all finite joins and meets. Define $\nabla I=\{f \in \mathcal{B} \mid \exists g \in I \; [f \leq s(g)]\}$ and $F(I)=\{f \in \mathcal{B} \mid \exists g \in I \, [f \leq (x_0 \to_{\mathcal{B}} s(g))]\}$. Note that both $\nabla$ and $F$ are order-preserving. It is easy to see that $\nabla$ is also join preserving. As $\mathcal{H}$ is complete, $\nabla$ has a right adjoint $\Delta$. Define $I \to_{\mathcal{H}} J=\Delta (F(I) \Rightarrow F(J))$, where $\Rightarrow$ is the Heyting implication of $\mathcal{H}$. The operation $\to_{\mathcal{H}}$ is an implication, by Example \ref{ExampleOfImplications}. By the adjunction $\nabla \dashv \Delta$, it is clear that $\nabla K \cap F(I) \subseteq F(J)$ iff $K \subseteq I \to_{\mathcal{H}} J$, for any ideal $I, J, K$ of $\mathcal{B}$. The only remaining thing to prove is the identity $i(a \to_{\mathcal{A}} b)=i(a) \to_{\mathcal{H}} i(b)$, for any $a, b \in \mathcal{A}$.
To prove, it is enough to show that $\nabla I \cap F(i(a)) \subseteq F(i(b))$ iff $I \subseteq i(a \to_{\mathcal{A}} b)$, for any ideal $I$ of $\mathcal{B}$. For the forward direction, if $f \in I$, then $s(f) \in \nabla I$ and as $c_a \in i(a)$, we have $s(f) \wedge (x_0 \to_{\mathcal{B}} s(c_a)) \in \nabla I \cap F(i(a))$. Therefore, $s(f) \wedge (x_0 \to_{\mathcal{B}} s(c_a)) \in F(i(b))$ which implies $s(f) \wedge (x_0 \to_{\mathcal{B}} s(c_a)) \leq (x_0 \to_{\mathcal{B}} s(c_b))$. This means, $f(\{x_i\}_{i=1}^{\infty}) \wedge (x_0 \to_{\mathcal{A}} a) \leq x_0 \to_{\mathcal{A}} b$, for any $\{x_i\}_{i=0}^{\infty}$.
Putting $x_0=a$, we reach $f(\{x_i\}_{i=1}^{\infty}) \leq a \to_{\mathcal{A}} b$, for any $\{x_i\}_{i=0}^{\infty}$. Thus, as $\{x_i\}_{i=0}^{\infty}$ is arbitrary, we have $f(\{x_i\}_{i=0}^{\infty}) \leq a \to_{\mathcal{A}} b$, for any $\{x_i\}_{i=0}^{\infty}$. Therefore, $f \leq c_{a \to_{\mathcal{A}} b}$ which implies $f \in i(a \to_{\mathcal{A}} b)$. For the backward direction, if $f \in \nabla I \cap F(i(a))$, then there exist $g \in I$ and $h \in i(a)$ such that $f \leq s(g) \wedge (x_0 \to_{\mathcal{B}} s(h))$. As $g \in I$, we have $g \leq c_{a \to_{\mathcal{A}} b}$. Hence, $f \leq c_{a \to_{\mathcal{A}} b} \wedge (x_0 \to_{\mathcal{B}} c_a)$ which means $f(\{x_i\}_{i=0}^{\infty}) \leq (a \to_{\mathcal{A}} b) \wedge (x_0 \to_{\mathcal{A}} a)$, for any $\{x_i\}_{i=0}^{\infty}$. Therefore, $f(\{x_i\}_{i=0}^{\infty}) \leq x_0 \to_{\mathcal{A}} b$ which implies $f \leq x_0 \to_{\mathcal{B}} s(c_b)$. Hence, $f \in F(i(b))$.\\
Finally, for the preservation of the openness, if $\to_{\mathcal{A}}$ is open, using Lemma \ref{PropertiesOfOpenAndClosed}, it is enough to prove $I \subseteq J \to_{\mathcal{H}} I \cap J$, for any $I, J \in \mathcal{H}$. Using the adjunction-style condition for $\to_{\mathcal{H}}$, we must show that
$\nabla I \cap F(J) \subseteq F(I \cap J)$, for any $I, J \in \mathcal{H}$. To prove it, let $f \in \nabla I \cap F(J)$. Hence, there are $g \in I$ and $h \in J$ such that $f \leq s(g)$ and $f \leq x_0 \to_{\mathcal{B}} s(h)$, i.e., $f(\{x_i\}_{i=0}^{\infty}) \leq g(\{x_i\}_{i=1}^{\infty})$ and $f(\{x_i\}_{i=0}^{\infty}) \leq x_0 \to_{\mathcal{A}} h(\{x_i\}_{i=1}^{\infty})$, for any $\{x_i\}_{i=0}^{\infty}$. We want to show that 
\[
f(\{x_i\}_{i=0}^{\infty}) \leq x_0 \to_{\mathcal{A}} [g(\{x_i\}_{i=1}^{\infty}) \wedge h(\{x_i\}_{i=1}^{\infty})].
\]
As $\to_{\mathcal{A}}$ is open, $g(\{x_i\}_{i=1}^{\infty}) \leq h(\{x_i\}_{i=1}^{\infty}) \to_{\mathcal{A}} g(\{x_i\}_{i=1}^{\infty}) \wedge h(\{x_i\}_{i=1}^{\infty})$. Therefore, $f(\{x_i\}_{i=0}^{\infty})$ is less than or equal to
\[
(x_0 \to_{\mathcal{A}} h(\{x_i\}_{i=1}^{\infty})) \wedge (h(\{x_i\}_{i=1}^{\infty}) \to_{\mathcal{A}} [g(\{x_i\}_{i=1}^{\infty}) \wedge h(\{x_i\}_{i=1}^{\infty})]).
\]
Hence, $f(\{x_i\}_{i=0}^{\infty}) \leq x_0 \to_{\mathcal{A}} [g(\{x_i\}_{i=1}^{\infty}) \wedge h(\{x_i\}_{i=1}^{\infty})]$. Now, set $k=g \wedge h$. Therefore, $k \in I \cap J$ and $f \leq x_0 \to_{\mathcal{B}} s(k)$. Hence, $f \in F(I \cap J)$.
\end{proof}
In the following, we define a topological version of the KN-frames introduced before to represent all (open) strong algebras. Apart from providing a more specific representation, the topology helps to represent the open strong algebras by the open KN-frames. Later, we will see that to represent an arbitrary implication (without respecting the openness condition), even the full KN-frames (without using any topology) are sufficient.  
\begin{definition}
A \emph{KN-space} is a KN-frame $\mathcal{X}=(X, \leq, R, C(X), N)$, where $(X, \leq)$ is a Priestley space and $C(X)$ is the set of all clopens of $X$. Spelling out, it means that:
\begin{itemize}
\item[$\bullet$]
$R$ is compatible with the order, i.e., if $x \leq y$ and $(y, z) \in R$, then $(x, z) \in R$,
\item[$\bullet$]
for any $x \in X$ and any clopen upsets $U$ and $V$, if $U \subseteq V$ and $U \in N(x)$ then $V \in N(x)$,
\item[$\bullet$]
$\lozenge_R(U)=\{x \in X \mid \exists y \in U, (x, y) \in R\}$ is clopen, for any clopen $U$,
\item[$\bullet$]
$j(U)=\{x \in X \mid U \in N(x)\}$ is clopen, for any clopen upset $U$.
\end{itemize}
\end{definition}

\begin{theorem} \label{RepTheorem} (KN-space Representation)
For any strong algebra $(\mathcal{A}, \to_{\mathcal{A}})$, there exist a KN-space $\mathcal{X}$ and an embedding $i : (\mathcal{A}, \to_{\mathcal{A}}) \to \mathfrak{A}(\mathcal{X})$. Moreover, if $\mathcal{A}$ is open, then so is $\mathcal{X}$. 
\end{theorem}
\begin{proof}
By Lemma \ref{RepresentationLemma}, w.l.o.g., we assume that $\mathcal{A}$ is a locale and there are join preserving map $\nabla: \mathcal{A} \to \mathcal{A}$ and order-preserving map $F: \mathcal{A} \to \mathcal{A}$ such that $\nabla c \wedge F(a) \leq F(b)$ iff $c \leq a \to_{\mathcal{A}} b$, for any $a, b, c \in \mathcal{A}$. Set $(X, \leq)$ as the Priestley space $(\mathcal{F}_p(\mathcal{A}), \subseteq)$ and define $R$ as $\{(P, Q) \in X^2 \mid \nabla[P] \subseteq Q\}$. Set $i(a)=\{P \in X \mid a \in P\}$ and define $N(P)=\{U \in CU(X, \leq) \mid \exists a \in \mathcal{A} \; [i(a) \subseteq U \; \text{and} \; F(a) \in P]\}$. First, it is clear that $R$ is compatible with the inclusion, $N(P)$ is upward closed on the clopen upsets of $(X, \leq)$, for any $P \in X$ and $i$ is a bounded lattice embedding. Secondly, observe that $i(a) \in N(P)$ iff $F(a) \in P$, for any $a \in \mathcal{A}$ and $P \in \mathcal{F}_p(\mathcal{A})$. One direction is obvious from the definition of $N$. For the other, if $i(a) \in N(P)$, then there exists $b \in \mathcal{A}$ such that $F(b) \in P$ and $i(b) \subseteq i(a)$. Hence, $b \leq a$ which implies $F(b) \leq F(a)$. As $P$ is a filter, $F(a) \in P$. Now, to prove that $j$ maps the clopen upsets to the clopens, if $U$ is a clopen upset, there is $a \in \mathcal{A}$ such that $U=i(a)$. Therefore, $j(U)=j(i(a))=\{P \in X \mid i(a) \in N(P)\}=i(F(a))$ which is clopen.\\
For the closure of the clopens under $\Diamond_R$, 
we need an auxiliary implication. First, notice that $\mathcal{A}$ is a locale and $\nabla$ is join preserving. Hence, it has a right adjoint $\Delta$. Define $a \to_{\nabla} b=\Delta(a \Rightarrow b)$, where $\Rightarrow$ is the Heyting implication of $\mathcal{A}$. By Example \ref{ExampleOfImplications}, it is clear that $\to_{\nabla}$ is an implication. Also notice that $c \leq a \to_{\nabla} b$ iff $\nabla c \wedge a \leq b$, for any $a, b, c \in \mathcal{A}$. The important property of $\to_{\nabla}$ for us is that $\Diamond_R(i(a) \cap i(b)^c)=i(a \to_{\nabla} b)^c$, for any $a, b \in \mathcal{A}$. To prove it, we have to show that $a \to_{\nabla} b \in P$ iff for any $Q \in X$ such that $\nabla[P] \subseteq Q$, if $a \in Q$ then $b \in Q$. The forward direction is easy, as $a \to_{\nabla} b \in P$ implies $\nabla(a \to_{\nabla} b) \in Q$. Hence, if $a \in Q$, as $a \wedge \nabla(a \to_{\nabla} b) \leq b$, we reach $b \in P$. For the converse, if $a \to_{\nabla} b \notin P$, then define $G$ and $I$ as the filter generated by $\nabla[P] \cup \{a\}$ and the ideal generated by $b$, respectively. It is clear that $I \cap G=\varnothing$. As otherwise, if $x \in I \cap G$, there are $p_1, \ldots, p_n \in P$ such that $\bigwedge_{i=1}^n \nabla p_i \wedge a \leq x \leq b$. As $\nabla$ is order-preserving, $\nabla(\bigwedge_{i=1}^n p_i) \wedge a \leq b$ which implies $\bigwedge_{i=1}^n p_i \leq a \to_{\nabla} b$. As $p_i$'s are in $P$ and $P$ is a filter, $a \to_{\nabla} b \in P$ which is a contradiction. Hence, $G \cap I=\varnothing$. Now, by the prime filter theorem, there is a prime filter $Q$ such that $G \subseteq Q$ and $Q \cap I=\varnothing$. By the former, we see that $\nabla[P] \subseteq Q$ and $a \in Q$. By the latter, we see $b \notin Q$. \\
Now, having the identity $\Diamond_R(i(a) \cap i(b)^c)=i(a \to_{\nabla} b)^c$, we prove the closure of the clopens of $X$ under $\Diamond_R$. Assume that $U$ is clopen. Then, there are finite sets $\{a_1, \ldots, a_n\}, \{b_1, \ldots, b_n\} \subseteq \mathcal{A}$ such that $U=\bigcup_{r=1}^n i(a_r) \cap i(b_r)^c$. As $\Diamond_R$ preserves the unions, we have $\Diamond_R( U)=\bigcup_{r=1}^n \Diamond_R(i(a_r) \cap i(b_r)^c)$. As $\Diamond_R(i(a_r) \cap i(b_r)^c)=i(a_r \to_{\nabla} b_r)^c$, we conclude that $\Diamond_R(i(a_r) \cap i(b_r)^c)$ and hence $\Diamond_R( U)$ is clopen.\\
The only remaining part is showing $i(a \to_{\mathcal{A}} b)=i(a) \to_{\mathcal{X}} i(b)$, for any $a, b \in \mathcal{A}$. To prove 
$i(a \to_{\mathcal{A}} b) \subseteq i(a) \to_{\mathcal{X}} i(b)$, assume $P \in i(a \to_{\mathcal{A}} b)$ which implies $a \to_{\mathcal{A}} b \in P$. Assume $(P, Q) \in R$ and $i(a) \in N(Q)$, for some $Q \in X$. By definition of $R$, we have $\nabla[P] \subseteq Q$. As we observed above, $i(a) \in N(Q)$ implies $F(a) \in Q$. Hence, we also have $F(a)\in Q$. To prove $F(b) \in Q$, as $a \to_{\mathcal{A}} b \in P$, we have $\nabla(a \to_{\mathcal{A}} b) \in \nabla[P] \subseteq Q$. As $Q$ is a filter and $\nabla(a \to_{\mathcal{A}} b) \wedge F(a) \leq F(b)$, we reach $F(b) \in Q$.
For the converse, assume $a \to_{\mathcal{A}} b \notin P$. We intend to provide a prime filter $Q \in X$ such that $\nabla[P] \subseteq Q$ and $F(a) \in Q$ but $F(b) \notin Q$. Set $I$ and $G$ as the ideal generated by $F(b)$ and the filter generated by $\nabla[P] \cup \{F(a)\}$, respectively. We claim $I \cap G=\varnothing$. Otherwise, if $x \in I \cap G$, there are $p_1, \ldots, p_n \in P$ such that $\bigwedge_{i=1}^n \nabla p_i \wedge F(a) \leq x \leq F(b)$. As, $\nabla$ is order-preserving, $\nabla(\bigwedge_{i=1}^n p_i) \wedge F(a) \leq F(b)$ which implies $\bigwedge_{i=1}^n p_i \leq a \to_{\mathcal{A}} b$. As $p_i$'s are in $P$ and it is a filter, $a \to_{\mathcal{A}} b \in P$ which is a contradiction. Hence, $G \cap I=\varnothing$. Now, by the prime filter theorem, there is a prime filter $Q \in X$ such that $G \subseteq Q$ and $Q \cap I=\varnothing$. By the former, we reach $\nabla[P] \subseteq Q$ and $F(a) \in Q$. By the latter, we prove $F(b) \notin Q$.\\
Finally, note that as all clopen upsets are in the image of $i$, if $\to_{\mathcal{A}}$ is open, for any clopen upsets $U$ and $V$, there are $a, b \in \mathcal{A}$ such that $U=i(a)$ and $V=i(b)$. Therefore, as $a \leq b \to_{\mathcal{A}} a \wedge b$ and $i$ preserves the meet and the implication, we reach $U \subseteq V \to_{\mathcal{X}} U \cap V$. Hence, 
for any $P, Q \in X$ and any clopen upsets $U$ and $V$, if $P \in U$, $(P, Q) 
\in R$ and $V \in N(Q)$, then $U \cap V \in N(Q)$.
\end{proof}

\begin{corollary}
For any strong algebra $(\mathcal{A}, \to_{\mathcal{A}})$, there exist a full KN-frame $\mathcal{K}$ and an embedding $i : (\mathcal{A}, \to_{\mathcal{A}}) \to \mathfrak{A}(\mathcal{K})$.
\end{corollary}
\begin{proof}
It is an easy consequence of Theorem \ref{RepTheorem} and Remark \ref{ThePassageToFull}.
\end{proof}

\noindent \textbf{Acknowledgments:} The support by the FWF project P 33548 and the Czech Academy of Sciences (RVO 67985840) are gratefully acknowledged.

\bibliographystyle{plainurl}
\bibliography{newbib}

\end{document}